\documentclass[12pt]{article}
\usepackage{amssymb}
\usepackage{amsmath}
\usepackage{amsthm}
\usepackage{graphicx}
\usepackage{psfrag}
\usepackage{graphics}
\usepackage{amsthm}
\usepackage{hyperref,enumitem,bbold}
\usepackage[dvipsnames]{xcolor}
\usepackage[english]{babel}

\newcommand{\cS}{\mathcal{S}}

\newcommand{\cM}{\mathcal{M}}
\newcommand{\cD}{\mathcal{D}}

\newcommand{\N}{\mathbb N}

\newcommand{\T}{\mathbb T}
\newcommand{\sS}{\mathbb S}

\newcommand{\var}{\mathrm{var}}

\newcommand{\drift}{\mathrm{drift}}

\newcommand{\R}{\mathbb{R}}

\newcommand{\Z}{\mathbb Z}

\newcommand{\cC}{\mathcal C}

\newcommand{\Diff}{\mathrm{Diff}}

\newcommand{\eps}{\varepsilon}

\def\Fix{{\mathrm {Fix}}}

\def\norm#1{\left\lVert #1 \right\rVert}

\theoremstyle{theorem}

\newtheorem{thm}{Theorem}[section]
\newtheorem*{thmA}{Theorem A}
\newtheorem*{thmB}{Theorem B}
\newtheorem*{thmC}{Theorem C}

\newtheorem*{thmF}{Theorem F}

\newtheorem{thmintro}{Theorem}

\newtheorem{qsintro}{Question}

\newtheorem{prop}[thm]{Proposition}

\newtheorem{cor}[thm]{Corollary}

\newtheorem{qs}[thm]{Question}

\newtheorem{lem}[thm]{Lemma}
\newtheorem*{lemsn}{Lemma}

\theoremstyle{definition}
\newtheorem{defn}[thm]{Definition}
\newtheorem{notation}[thm]{Notation}

\newtheorem{claim}[thm]{Claim}
\theoremstyle{remark}
\newtheorem{rem}[thm]{Remark}

\newcommand{\id}{\mathrm{id}}
\newcommand{\f}{\varphi}

\begin{document}

\date{}
\vspace{-1cm}

\date{}
\author{H\'el\`ene Eynard-Bontemps \,\, \& \,\, Andr\'es Navas}

\title{\vspace{-1.5cm}How (not) to prove (un)distortion \\
for diffeomorphisms of one-manifolds
}

\maketitle

\noindent{\bf Abstract.} This article addresses the following general question: Given a one-dimensional manifold $M$ and $1\le r<s\le \infty$, does there exist a $C^s$ orientation preserving compactly supported diffeomorphism of $M$ that is undistorted in the group $\Diff_{c,+}^s(M)$ of such diffeomorphisms while distorted in the bigger group of $C^r$ diffeomorphisms? Interestingly, the answer is known to be positive in the case $(r,s)=(1,2)$ and negative in the case $(r,s)=(2,\infty)$, according to \cite{Na21} and \cite{EBM25}, respectively. 

The first part of this note originates from a failed attempt to extend the ideas of \cite{Na21} to the case $(r,s)=(2,3)$. More precisely, in regularities $C^1$ and $C^2$, obstructions to distortion are provided by \emph{drifts of cocycles for isometric actions of $\Diff_{c,+}^r(M)$ on Banach spaces for $r=1$ and $r=2$ (namely, the logarithmic and projective derivatives $f\mapsto \log Df$ and $f\mapsto D\log Df$, respectively)}. On $\Diff_{c,+}^3(M)$, the so-called \emph{Liouville cocycle} is a natural candidate when looking for new obstructions, but we show that its drift vanishes for $C^2$-distorted diffeomorphisms (and this holds more generally for any ``similar'' cocycle). 

This does not rule out the existence of $C^2$-distorted diffeomorphisms that are $C^3$-undistorted. 
However, at least in the case of the real line, such a diffeomorphism should have very low regularity. 
Indeed, extending the methods and results of \cite{EBM25}, in the second part of this article, 
we show that every compactly supported $C^2$-distorted diffeomorphism of the real line is $C^r$-distorted provided its differentiability class is larger than $C^{2r+2}$.\\

\tableofcontents

\section{Introduction}

In this work, we are interested in two properties that a diffeomorphism of a one-manifold might or might not have, namely: \emph{distortion} and \emph{almost-reducibility}. Distortion is a very general geometric group theoretic notion: an element $f$ of a group $G$ is said to be \emph{distorted} in $G$ if there exists a finite subset $S\subset G$ such that $f\in\langle S\rangle$ and the word-length $|f^n|_S$ of the iterates of $f$ with respect to $S$ grows sublinearly in $n$. Almost-reducibility, on the other hand, concerns topological groups. In our context, we will say that a compactly supported $C^r$ diffeomorphism of a one-manifold $M$ (specifically, $M=\R$, $[0,1]$ or $\sS^1$) is $C^r$-almost-reducible if it can be $C^r$-conjugated arbitrarily close to an isometry in $C^r$ topology by diffeomorphisms with common compact support. Concretely, if $M$ is an interval or the real line, the only compactly supported isometry of $M$ is the identity, while if $M$ is the circle, the isometries are the rotations. 

One interesting aspect of these notions is that they highlight deep differences between the groups $\Diff^r_{c}(M)$ for different values of $r$, where the subindex ``c'' stands for ``compactly supported. A key instance of this is the construction by the second-named author in \cite{Na21} of a smooth diffeomorphism of the closed interval which is ``$C^1$-distorted'' (that is, distorted in the group of $C^1$ diffeomorphisms) and $C^1$-almost-reducible but neither $C^2$-distorted nor $C^2$-almost-reducible. (See \cite{dinamarca-escayola} for a subsequent smoothing of this example to class $C^{1+\alpha}$.) This naturally leads to the following question, also from \cite{Na21}:

\begin{qsintro}
\label{q:intro} Given a manifold $M$ and integers $0\le r<s$, does there exist a $C^s$
diffeomorphism that is distorted in $\Diff^r_{c}(M)$ while undistorted in $\Diff^s_c(M)$?
\end{qsintro}

One of the purposes of this article is to present results that point towards a negative answer in the one-dimensional case for higher values of $r,s$. To simplify the discussion, we will only consider orientation-preserving diffeomorphisms. The general case can be easily deduced from this taking into account a couple of minor considerations. First, an orientation-reversing diffeomorphism $f$ is a distortion element if and only if $f^2$ is; second, a simple trick shows that, in this situation, $f^2$ is distorted in the whole group of diffeomorphisms if and only if it is distorted in the group of orientation-preserving ones. Indeed, it is enough to change the corresponding 
generating system $S$ by the set $S'$ formed by the elements of the form $g \varphi^{\pm 1}$ and $\varphi^{\pm 1} g$ if $g \in S$ inverts orientation, and $g$ and $\varphi^{\pm 1} g \varphi^{\pm 1}$ if $g \in S$ preserves orientation, where $\varphi$ denotes any orientation-reversing diffeomorphism.

The results of \cite{EBM25} already give a negative answer to Question A for the case $M=\R$ or $\sS^1$, $r=2$ and $s=\infty$. Indeed, it follows from \cite{Na23} that if a smooth diffeomorphism of $M$ is $C^2$-distorted, then it must have vanishing \emph{asymptotic variation} (see below), and \cite{EBM25} proves that this is a sufficient condition for being $C^\infty$-almost-reducible and distorted in $\Diff^\infty_c(\R)$ or $\Diff^\infty_+(\sS^1)$. (Let us point out that $C^{\infty}$-almost-reducibility of minimal circle diffeomorphisms corresponds to a --still unpublished-- result of Avila and Krikorian.) Here, we focus on the finite regularity case. 

This article can be subdivided in two distinct parts: the first one is formed by Sections \ref{s:cocycles} to~\ref{s:necessary}, and the second one by Sections~\ref{s:a-r} to \ref{s:perfect}. In all what follows, $M$ will denote either the closed interval or the circle.

\paragraph{Overview of the first part.} The first three sections deal will a classical tool to prove undistortion, namely (stabilized) length functions, and more specifically those coming from \emph{cocycles} for isometric actions of $\Diff^r_+(M)$ on Banach spaces. Rather than providing a precise definition here (see Section \ref{s:cocycles} for 
the details), let us illustrate the general idea with an example for $\Diff^1_+(M)$. In this context, the logarithmic derivative has a nice chain rule: for all $f,g \in\Diff^1_+(M)$,
$$\log D(f\circ g)=\log (Df) \circ g+\log Dg.$$
Note that given $g\in\Diff^1_+(M)$, the map $U(g):u\mapsto u\circ g$ defines an isometry of $L^\infty(M)$ (or $C^0(M)$ endowed with the $\|\cdot\|_\infty$ norm). If $\delta$ denotes the logarithmic derivative operator $f\mapsto \log Df$, the above equality can be rewritten as 
$$ \delta(f\circ g)=U(g) (\delta(f)) + \delta(g).$$
We can summarize this by saying that \emph{$\delta$ is a cocycle for the isometric action of $\Diff^1(M)$ on $L^\infty(M)$}. Now, the above implies that 
$$\|\delta(f\circ g)\|_\infty\le \|\delta(f)\|_\infty+\|\delta(g)\|_\infty,$$
which we rephrase by saying that $\|\delta(\cdot)\|_\infty$ \emph{is a length function on $\Diff^1_+(M)$}. 
By the classical Fekete Lemma on subadditive sequences, this implies that the quotient $\frac{\|\delta(f^n)\|_\infty}n$ has a finite limit, which we call the \emph{drift of $\delta$ at $f$} and simply denote by $\drift_\delta(f)$. 

It is well-known (and easy to check) that if an element $f$ in a group $G$ is distorted, then for any length function $L$ on $G$, the quotient $\frac{L(f^n)}n$ must go to $0$ as $n$ goes to infinity. (Here and in what follows, by a {\em length function on a group $G$ we mean a function $L: G \to \mathbb{R}_+$ such that $L(fg) \leq L(f) + L(g)$ holds for all $f,g$ in $G$.}) In particular, in our context, if $f\in\Diff^1(M)$ is distorted in this group, then $\drift_\delta(f)=0$. Luckily, this drift is easy to interpret: it corresponds to the maximum of $| \! \log Df^{p}(a)| / p$ where $a$ ranges over the periodic points of $f$ and $p$ is the period (which is the same for all periodic points). In particular, if a diffeomorphism is $C^1$-distorted, then it cannot have any hyperbolic periodic point. 
However, it is unknown whether the converse is true.

One strategy to attack Question \ref{q:intro} is to try to construct new length functions in high regularity that could provide new obstructions to distortion (in the corresponding regularity). This is done in \cite{Na23} for $\Diff^2_+(M)$ (actually, for the bigger group $\Diff^{1+ac}_+(M)$ of $C^1$ diffeomorphisms with absolutely continuous derivative) using the ``projective derivative''  $P:g\mapsto D\log Dg$. Indeed, this is a cocycle for the canonical action of $\Diff^2_+(M)$ by isometries of $L^1(M)$ given by $U(g)(u) = (u\circ g) Dg$. The drift in this case is called the 
 ``asymptotic variation of $f$'',  and is denoted by $V_\infty(f)$ (with our current notations, this would be $\drift_P(f)$). 
 In simpler terms, $P$ satisfies the chain rule: 
$$P(f\circ g)=(P(f)\circ g) \times Dg +P(g),$$ 
and $V_\infty$ is the stabilization of the length function $f \in\Diff^2_+(M)\mapsto \| P(f) \|_{L^1}$:
$$ V_\infty(f) =\drift_P(f) = \lim_{n\to+\infty}\frac{\|D\log Df^n\|_{L^1}}n\;\in\R_+.$$
Note that, again, this number has a nice interpretation in the case of the interval: for a diffeomorphism without interior fixed point and parabolic at the boundary, it is equal to $\var\log DM_f$, where $M_f$ denotes the \emph{Mather invariant of $f$} (cf. \cite{EN21} and \S\ref{ss:drift-mu}), which quantifies the ``non embeddability'' of $f$ in a $C^1$ flow. 

With the terminology above, the example built in \cite{Na21} can be described as a smooth diffeomorphism of the interval that is $C^1$-distorted (which is proved by hand) but has non-vanishing asymptotic variation and is thus $C^2$-undistorted.

In the later article \cite{DN}, the authors suggest to study the drift of the so-called \emph{Liouville cocycle} 
(cf. Section \ref{s:drift}) defined on the group $\Diff^{>2}_c(M)$ of $C^2$ diffeomorphisms with Hölder continuous second derivative, hoping to yield new obstructions there. Concretely, for a diffeomorphism $f$ of this regularity, 
$\ell(f)$ is the $L^1$ function on $[0,1]^2$ defined (outside the diagonal) by
$$\ell(f)(x,y) = \frac{Df(x) Df(y)}{d(f(x),f(y))^2} - \frac{1}{d(x,y)^2},$$
{where $d$ denotes the usual distance on $M$. This map $\ell$ defines a cocycle for the canonical action of 
$\Diff^{>2}_c(M)$ by isometries of $L^1([0,1]^2)$ (see \cite{Na06}, or Section \ref{s:cocycles}), and is closely related to the classical Schwarzian derivative. However, the first result of this article shows that this line of reasoning leads to a dead-end and that, more generally, no $\drift$ of any cocycle for the standard isometric action of $\Diff^r_+(M)$ on $L^1(M^d)$ ($r\ge2$, $d \in \mathbb{N}^*$) will yield new obstructions, at least for diffeomorphisms without interior fixed points: 

\begin{thmintro}
\label{t:vanish-drift-gen}
Let $r\ge2$ and $c:\Diff^r_+([0,1])\to L^1([0,1]^d)$ be a continuous map such that, for every $f,g\in\Diff^r_+([0,1])$, 
the following equality is satisfied for almost every $(x_1,\dots,x_d)\in[0,1]^d${\em :}
$$c(f\circ g)(x_1,\dots,x_d)= c(f)(g(x_1),\dots,g(x_d))Dg(x_1)\dots Dg(x_d)+c(g)(x_1,\dots,x_d)$$
(i.e. $c$ is a continuous cocycle for the canonical isometric action of $\Diff^r_+([0,1])$ on $L^1([0,1]^d)$). 
If $f\!\in\!\Diff^{r}_+([0,1])$ has vanishing asymptotic variation and no interior fixed point, then $\drift_c(f)\!=\!0$.
\end{thmintro}

Section \ref{s:cocycles} presents the necessary general background on cocycles and drifts as well as concrete examples where we clearly know what the drift quantifies dynamically. This gives us the opportunity to revisit some results and proofs from \cite{EN21} alluded to above, about the relation between the asymptotic variation of a diffeomorphism $f$ without interior fixed points (that is, the drift of the projective derivative at $f$, as explained before) and its \emph{Mather invariant} $M_f$. This relation is clarified here by considering a slightly different (``composite'') cocycle:
$$\forall f\in\Diff^2([0,1]),\quad \mu(f) := D\log Df - \left(\tfrac{Df}f-\tfrac1{\id}\right)-\left(\tfrac{Df}{f-1}-\tfrac1{\id-1}\right)\in L^1([0,1]).$$

\begin{thmintro}
\label{t:mu}
The equality  $\drift_\mu(f)=\var\log DM_f$ holds for every $f\in\Diff_+^{2}([0,1])$ without interior fixed point.  
\end{thmintro}

The original result of \cite{EN21} was the inequality 
\begin{equation}
\label{e:ADMIN}
|\drift_P(f)-\var\log DM_f|\le |\log Df(0)|+|\log Df(1)|,
\end{equation}
(recall that $\drift_P(f)=V_\infty(f)$), which provides an equality only in the case where both endpoints are parabolic. The addition of the terms $\left(\tfrac{Df}f-\tfrac1{\id}\right)$ and $\left(\tfrac{Df}{f-1}-\tfrac1{\id-1}\right)$ in the ``mixed'' cocycle $\mu$ allows to get a general equality statement. In particular, we show in Section \ref{ss:drift-nu} that the drifts of these terms, which also define cocycles (first introduced in \cite[Chapter 4]{Na03}), correspond precisely to the terms $|\log Df(0)|$ and $|\log Df(1)|$. Note, however, that the inequality (\ref{e:ADMIN}) above holds more generally for $C^{1+bv}$ diffeomorphisms, while this regularity is too weak to ensure the $L^1$ character of $\mu(f)$.\medskip

Section \ref{s:drift} is devoted to the proof of Theorem \ref{t:vanish-drift-gen}. It relies on a localization formula for the drift at a diffeomorphism without interior fixed point. This allows us to prove the continuity of the drift at such diffeomorphisms (we will see in \S\ref{ss:non-cont}, however, that the drift is not globally continuous in general, as was already observed in \cite[Theorem C]{EN21} for the projective derivative). Finally, to prove the theorem, we show that any diffeomorphism satisfying its assumptions can be approximated by ``nice'' diffeomorphisms whose drift is known to vanish. \medskip

Section \ref{s:necessary} is devoted to the Liouville cocycle introduced above. Although there is no simple interpretation for its drift (in contrast, for example, to the case of the ``corrected projective cocycle'' $\mu$ above), we have a simple characterization of when it vanishes:

\begin{thmintro}
\label{t:criterion}
Given $f\in \Diff^{>2}_+([0,1])$, the Liouville drift vanishes at $f$ if and only if $f$ is $C^{>2}$-conjugated to the restriction to $[0,1]$ of a Möbius map of $\R\cup\{\infty\}$ or $f$ 
embeds in a $C^1$ flow without hyperbolic fixed points (equivalently, has vanishing asymptotic variation).\end{thmintro}

One of the implications (the ``if'' part) follows, in the case without interior fixed points, from Theorem \ref{t:vanish-drift-gen} (and the fact that $\ell$ vanishes at Möbius maps and that a drift is always a conjugacy invariant, cf. Section \ref{s:drift}). The general case (with interior fixed points) uses a ``fragmentation formula'' proved in \S\ref{ss:fragment}.  The other implication (the ``only if'' part) comes from an explicit relation between all the cocycles mentioned so far (cf. Lemma \ref{l:relation} and Corollary \ref{c:necessary}). Things are much simpler in the case of the circle (cf. \S\ref{s:circle}), where we adapt an argument from \cite{Na23} regarding the asymptotic variation to obtain:

\begin{thmintro}
\label{t:cercle-intro}The Liouville drift at a $C^{>2}$ circle diffeomorphism with irrational rotation number always vanishes. 
\end{thmintro}

\paragraph{Overview of the second part.} On the one hand, stabilized length functions (and, in particular, drifts of cocycles considered in the first part) provide obstructions to distortion and almost-reducibility (this has not been emphasized so far, but they provide conjugacy invariants that are continuous at the identity, where they vanish). On the other hand, one would like to prove that distortion and almost-reducibility do hold if these obstructions 
vanish. This is partially carried out in the second part of this article, yet the results therein are incomplete, 
in the sense that they involve some loss of regularity (which is inherent to the involved tools). Here is what we get in Section \ref{s:a-r}:

\begin{thmintro}
\label{t:a-r-intro}
Let $2\le R\in\N$ and $f\in\Diff^{R}_+([0,1])$ with vanishing asymptotic variation. Then $f$ is $C^{r}$-almost-reducible for every $r<\frac{R}2$, 
and actually $C^{R-1}$-almost-reducible if $f-\id$ is nowhere $C^R$-flat. 
\end{thmintro}

Of course, this holds a fortiori for compactly supported diffeomorphisms of $\R$. This probably easily implies a similar statement for circle diffeomorphisms with rational rotation number. 
Also, there is probably a similar statement for circle diffeomorphisms with irrational rotation number, based on renormalization techniques (recall that in this context, an unpublished work by Avila and Krikorian claims that the statement above is true for $R=r=\infty$). 

The proof of Theorem~\ref{t:a-r-intro} is based on fine estimates on the generating vector fields of interval diffeomorphisms obtained in Section \ref{s:bounds} (which are refinements of results of \cite{Se,BE16} in finite regularity), and on interpolation techniques developed in \cite{EBM25} in the smooth case. 

\begin{rem}
\label{r:C1}
In this article, we are mostly dealing with regularity $C^2$ and higher. Regarding almost-reducibility, the $C^1$ case is fully understood: the absence of hyperbolic periodic points is an obvious necessary condition of $C^1$-almost-reducibility for interval or circle diffeomorphisms, and it turns out that this condition is also sufficient according to \cite{Na14} (see also \cite{Fa} for the interval case). 
\end{rem}

Finally, again adapting the techniques of \cite{EBM25}, we prove:

\begin{thmintro}
\label{t:dist-fini}
If $f\in\Diff^r_c(\R)$ is $C^r$-almost-reducible for a certain $r \geq 3$, then it is $C^{r-2}$-distorted. 
\end{thmintro}

The proof of this theorem requires some weak ``uniform local perfectness'' result proved in Section \ref{s:perfect}, 
which is again an adaptation in finite regularity of arguments of \cite{Av08} and \cite{EBM25}. The key ingredient of this result, a famous local linearization theorem for circle diffeomorphisms by Arnold, Moser and Herman, involves a ``loss of regularity'', which is naturally reverberated in the above statement. Note, moreover, that such a perfectness result, as well as other ingredients of the proof of Theorem \ref{t:dist-fini}, hold for compactly supported diffeomorphisms of $\R$ and not for diffeomorphisms of a fixed closed interval. This is why we had to slightly change the context in this last theorem, from $\Diff^r_+([0,1])$ to $\Diff^r_c(\R)$. Most probably, this is more than a technical issue, as it is related to the deep problem of understanding distortion in the group of germs of diffeomorphisms, that we plan to explore in a later work.
\medskip

We end this introduction with a few concrete and natural questions:

\begin{qs}
If $f\in \Diff^r_+([0,1])$, with $r\ge2$, has vanishing asymptotic variation, is $f$ $C^r$-almost-reducible? In other words: can one avoid the loss of regularity in Theorem \ref{t:a-r-intro}?
\end{qs}
As we will see, this loss of regularity is inherent to the techniques we employ, so some new idea is needed to answer this question. Note that in \cite{Na23} it is proved that the vanishing of $V_\infty(f)$ implies the $C^{1+ac}$-almost-reducibility of $f$, and that the conjugacies involved are actually $C^r$; however, it is unclear wether the convergence is stronger than $C^{1+ac}$.\medskip

\begin{qs}
If $f\in \Diff^r_c(\R)$ is $C^r$-almost-reducible, with $r\ge2$, is $f$ distorted in $\Diff^r_c(\R)$? In other words: can one avoid the loss of regularity in Theorem \ref{t:dist-fini}?
\end{qs}

This would require a uniform local perfection result like Theorem \ref{Thm:locfragperf} but without loss of regularity. However, again, the loss of regularity is inherent to the argument used there. 

\medskip

\paragraph{Notations.} In what follows, for $k\in\N$, $\alpha>0$, and $M=[0,1]$ or $\mathbb{S}^1$, 
$\Diff^{k+\alpha}_+(M)$ denotes the group of $C^k$ diffeomorphisms of $M$ with $\alpha$-Hölder $k$th derivative. Also, $\Diff^{>k}_+([0,1])$ denotes the group $\bigcup_{\alpha>0}\Diff^{k+\alpha}_+([0,1])$. We say that a map is $C^{>k}$ if it is $C^{k+\alpha}$ for some positive $\alpha$, that is, if its $k$th derivative is Hölder, without specifying the degree of ``Hölderness''.} 

We let $C^k(M)$ be the space of real-valued functions of class $C^k$ defined on $M$. For $u \in C^k (M)$, we let $\| u \|_k$ be its $C^k$ norm:
$$\| u \|_k := \max \{ \|D^i u\|_{\infty} : 0 \leq i \leq k \}.$$
The distance between $f$ and $g$ in $\Diff^{k}_+(M)$ is $d_{C^k} (f,g) := \| f - g \|_k.$ This defines a complete metric that endows $\Diff^{k}_+(M)$ with a topological group structure.

For $0 < \alpha \leq 1$, we let $C^{k+\alpha}(M)$ be the space of functions of class $C^k$ defined on $M$ whose 
derivatives are $\alpha$-H\"older continuous. For $u \in C^{k+\alpha} (M)$, we let $\| u \|_{k+\alpha}$ be its $C^{k+\alpha}$ norm:
$$\| u \|_{k+\alpha} := \max \Big\{ \|u\|_k, \max_{x \neq y} \frac{|D^ku(x)-D^ku(y)|}{|x-y|^{\alpha}} \Big\}.$$
The distance $d_{C^{k+\alpha}}(f,g)$ between $f$ and $g$ in $\Diff^{k+\alpha}_+(M)$ is $\| f- g \|_{k+\alpha}$.
This defines a complete metric that, however, does not endow $\Diff^{k+\alpha}_+(M)$ with a topological group structure. Indeed, although for every $g \in \Diff^{k+\alpha}_+(M)$, the (right translation) map $f \mapsto f \circ g$ is continuous, the (left translation) map $f \mapsto g \circ f$ fails to be continuous for certain $g \in \Diff^{k+\alpha}_+(M)$. We refer to \cite{Bri18} (Lemmas 3.4 and 3.5) for a neat description of this phenomenon; see also \cite[Appendix 4]{EN23}), where a similar phenomenon is exhibited in the group of diffeomorphisms having derivatives with bounded 
variation.
\medskip

Given $r\in\R_+$, we denote by $\Diff^{r,\Delta}_+([0,1])$ the subset of $\Diff^{r}_+([0,1])$ consisting 
of diffeomorphisms without interior fixed points. Note that this subspace is not closed.

Given a subset $S$ of a group $G$, the subgroup generated by $S$ will be denoted $\langle S \rangle$. 
An interval of the form $[\![a,b]\!]$ will denote the set of integers $n$ such that $a \leq n \leq b$. 
Finally, we denote $\N = \{0,1,\ldots\}$ and $\N_* = \{1,2,\ldots\}$.


\section{Cocycles and their drifts: definition and first examples}
\label{s:cocycles}

In this section, we introduce cocycles for isometric actions of a group on a Banach space and their drifts. We provide several examples of cocycles for the isometric action of $\Diff^r_+([0,1])$ on $L^1([0,1]^d)$, for different values of $r$ and $d$, and we give a concrete dynamical interpretation of their drifts.

\subsection{Generalities}
\label{ss:gen}

This subsection provides a synthetic overview of notions introduced in more details in \cite{EN21} and \cite{EN23}. 
Let $U$ be a linear isometric action of a group $G$ on a Banach space $\mathbb{B}$. A {\em cocycle} for $U$ is a map $c : G \to \mathbb{B}$ which, for all $g_1,g_2 $ in $G$, satisfies the relation
$$c (g_1 \circ g_2) = c(g_2) + U(g_2) (c(g_1))$$ 
(from which one gets in particular that $c(\id)=0$). Since $U$ is an isometric action, this implies the following triangle inequality:
$$\|c(g_1\circ g_2)\|_{\mathbb{B}}\le \|c(g_1)\|_{\mathbb{B}}+\|c(g_2)\|_{\mathbb{B}}.$$
As a consequence, for each $f \in G$, we can define {\em the drift of $c$ at $f$} as the nonnegative real number
$$\mathrm{drift}_c (f) := \lim_{n \to \infty} \frac{\| c(f^n) \|_{\mathbb{B}}}{n}\le \|c(f)\|_{\mathbb{B}}.$$ 

\begin{rem} 
In the setting above, we choose to use the contravariant formulation (equivalently, actions on the right) just to 
simplify the notation. Otherwise, we would be forced to systematically introduce the inverses of the maps, rather than the maps themselves, all along the expressions we manipulate.
\end{rem}

Regarding the previous remark, it is worth pointing out the following invariance property, which is important by itself:   
$$
\forall f \in G, \quad  \drift_c (f) = \drift_c (f^{-1}).
$$
Besides this, one can easily check that $\drift_c$ is invariant under conjugacy, which implies that
\begin{equation}
\label{e:upper}
\mathrm{drift}_c (f) \le \inf_{\varphi\in G}\|c(\varphi f \varphi^{-1})\|_{\mathbb{B}}.
\end{equation}
Furthermore, one can prove that
\begin{equation}
\label{eq-drift-cob}
\mathrm{drift}_c (f) = \inf_{u \in \mathbb{B} }\big\| c(f) - \big( u - U(f) (u) \big) \big\|_{\mathbb{B}}
\end{equation}
and, in particular, 
\begin{equation} 
\label{e:gen-fact}
\mathrm{drift_c}(f)=0 \quad\Leftrightarrow\quad c(f)\in\overline{\{u-U(f)(u)\,;u\in\mathbb{B}\}}.
\end{equation}

In this article, we will be almost exclusively interested in the case where $G=\Diff_+^r([0,1])$ for some given regularity 
$r\ge1$ and $\mathbb{B}=L^1([0,1]^d)$ for some $d\in\N^*$. Let us describe the natural isometric action of $G$ on 
$\mathbb{B}$ in this case. Let $\cM([0,1]^d)$ denote the space of measurable maps from $[0,1]^d$ to $\R$. 
Given $F\in\Diff_+^1([0,1]^d)$, we define the endomorphism $F^*$ of $\cM([0,1]^d)$ by:
$$F^*(u)= (u\circ F)\det(DF).$$
The notation is motivated by the standard pull-back operation on differential forms: 
$$F^*(u \, dx_1\wedge \dots \wedge dx_d)=(u\circ F) \det(DF) \, dx_1\wedge\dots\wedge dx_d$$
(in our context, functions are often viewed as things to integrate). Naturally, this operation has a good behavior 
with respect to composition: given $F,G\in\Diff_+^1([0,1]^d)$,
$$(FG)^*=G^*F^*.$$
By the change of variable formula, the endomorphism $F^*$ induces an isometry of $L^1([0,1]^d)$. We will be particularly interested in two cases:
\begin{itemize}
\item $d=1$: given $f\in\Diff_+^1([0,1])$, $f^*$ is simply:
$$f^*:u\mapsto (u\circ f )Df$$
\item $d=2$: given $f\in \Diff_+^1([0,1])$ again, if $f^{\otimes2}$ denotes the diffeomorphism $(x,y)\mapsto (f(x),f(y))$ of $[0,1]^2$, then $(f^{\otimes2})^*$ is the map:
 $$(f^{\otimes2})^*:u\mapsto \Big((x,y)\mapsto u(f(x),f(y))Df(x)Df(y)\Big).$$
 \end{itemize}
More generally, $U:f\mapsto (f^{\otimes d})^*$ defines a linear isometric action of $\Diff^1_+([0,1])$ on $L^1([0,1]^d)$. 
 
\subsection{Examples of cocycles}
\label{ss:def}

Let $u_0,u_1\in\cM([0,1])$ and $v\in\cM([0,1]^2)$ be the maps $x\mapsto \frac1x$, $x\mapsto \frac1{x-1}$ and $(x,y)\mapsto \frac1{(x-y)^2}$ respectively (neither of which is integrable). We define three ``differential operators'' 
$$\nu_0,\nu_1 : \Diff^1_+([0,1])\to \cM([0,1])$$ 
and 
$$\ell:\Diff^1_+([0,1])\to \cM([0,1])^2$$ 
by 
$$\nu_i(f)=f^*u_i-u_i= \left(x\mapsto \frac{Df(x)}{f(x)-i}-\frac1{x-i} \right)= D\log\left(\frac{f-i}{\id-i}\right)$$ 
and 
$$\ell(f) = (f^{\otimes2})^*v-v = \left( (x,y) \mapsto \frac{Df(x) Df(y)}{(f(x)-f(y))^2} - \frac{1}{(x-y)^2}\right)=\partial_x\partial_y \log\left(\frac{f(x)-f(y)}{x-y}\right).$$
From their definition, it immediately follows that these operators satisfy the following cocycle relations:
\begin{equation}
\label{e:coc-tau}
\nu_i(f\circ g)=g^*\nu_i(f)+\nu_i(g),
\end{equation}
\begin{equation}
\label{e:coc-c}
\ell(f\circ g)=(g^{\otimes2})^*\ell(f)+\ell(g).
\end{equation}

However, if we want them to take values in the nice Banach spaces $L^1([0,1])$ and $L^1([0,1]^2)$ we need to restrict them to $\Diff^{>1}_+([0,1])$ and $\Diff^{>2}_+([0,1])$ respectively, as we will see below. Before that, let us (re)introduce two more operators, namely the so-called projective and Schwarzian derivatives:
$$
P:\Diff^{2}([0,1])\to L^1([0,1]),\quad  P(f)=D\log Df
$$
(which actually extends to $C^1$ diffeomorphisms with absolutely continuous derivative) and 
$$
S:\Diff^{3}([0,1])\to C^0([0,1]),\quad S(f)=D^2\log Df-\tfrac12(D\log Df)^2.
$$
The usual chain rule implies that $P$ also satisfies the cocycle relation:
\begin{equation}
\label{e:coc-P}
P(fg)= g^*P(f)+P(g).
\end{equation}
The Schwarzian derivative has a chain rule as well, but not so nice in the perspective of \S\ref{ss:gen}:
$$
S(fg)= ( S( f )  \circ g ) \times (Dg)^2 + S(g).
$$

The content of the next two propositions is scattered throughout the literature, and we include them just for the sake of completeness. In the statements, the different spaces are endowed with their usual topology/metric.

\begin{prop}
\label{p:tau}
\begin{enumerate}
\item If $f\in \Diff^{2}_+[0,1]$, then $\nu_i(f)$, which is a priori only defined on $[0,1]\setminus\{i\}$, extends 
to a continuous map on $[0,1]$ by $\nu_0(f)(0)=\frac12P(f)(0)$ and $\nu_1(f)(1)=\frac12P(f)(1)$, respectively.
\item For every $\alpha>0$, $\nu_i$ defines a continuous map from $\Diff^{1+\alpha}_+([0,1])$ to $L^1([0,1])$.
\end{enumerate}
\end{prop}

\begin{prop}
\label{p:ell}
\begin{enumerate}
\item If $f\in \Diff^{3}_+[0,1]$, then $\ell(f)$, which is a priori not defined on the diagonal, extends to a continuous map on $[0,1]^2$ by $\ell(f)(x,x)=\frac16Sf(x)$ for every $x\in[0,1]$.  
\item For every $\alpha>0$, $\ell$ defines a continuous map from $\Diff^{2+\alpha}_+([0,1])$ to $L^1([0,1]^2)$.
\end{enumerate}
\end{prop}

\begin{proof}[Proof of Proposition \ref{p:tau}] We treat the case $i=0$, the other one being analogous.
\medskip

\noindent 1. Since $f(0)=0$, the Taylor integral formula gives:
$$
\forall x\in(0,1],\quad \frac{f(x)}x=\int_0^1 Df(tx)dt.
$$
If $f$ is $C^2$, the map $h:=\frac{f}{\id}$ thus extends to a (non-vanishing) $C^1$ function on $[0,1]$. 
Hence, $\nu_0(f)=D\log(\frac{f}{\id})$ extends to a continuous map on $[0,1]$. More precisely, on $(0,1]$, by ``derivation under the integral'',
\begin{align*}
(D\log h)(x) = (\tfrac{Dh}h)(x)=\frac{\int_0^1tD^2f(tx)dt}{\int_0^1Df(tx)dt}\xrightarrow[x\to0]{}\tfrac12Pf(0).
\end{align*}
\medskip

\noindent 2. Let $\alpha>0$. For $f\in\Diff^{1+\alpha}_+([0,1])$, let $C_f>0$ be the smallest constant such that, for all $x\in[0,1]$, 
$$
|Df(x)-Df(0)|=|D(f-\id)(x) - D(f-\id)(0)|\le C_f x^\alpha.
$$
Note that $\nu_0(f)$ is $C^0$ on $(0,1]$. Moreover,  
$$
|\nu_0(f)(x)| = \frac1{f(x)}\left|Df(x)-\int_0^1 Df(tx)dt\right| = \frac1{f(x)}\left|(Df(x)-Df(0))+\int_0^1 (Df(0)-Df(tx)) \, dt\right|\!,
$$
hence 
$$
|\nu_0(f)(x)| \le C_f\frac{x^\alpha}{f(x)}\left(1+\int_0^1t^\alpha dt\right).
$$
Since $\frac{x^\alpha}{f(x)}=O(\frac1{x^{1-\alpha}})$, it follows that $\nu_0(f)$ is integrable around $0$. 
Moreover, if $f$ goes to the identity in the $C^{1+\alpha}$ topology, then $C_f$ goes to $0$, and this implies that $\|\nu_0(f)\|_{L^1}$ goes to $0$. This shows the continuity of $\nu_0$ at the identity.

The continuity everywhere can be deduced from the cocycle identity as follows. Assume $(f_n)$ converges to $f$ in $\Diff^{1+\alpha}_+([0,1])$. By continuity of the right translation, $(f_nf^{-1})$ goes to $ff^{-1}=\id$, and thus $\nu_0(f_nf^{-1})$ goes to $0$ according to the above. Now, 
$$
\nu_0(f_n)=\nu_0(f_nf^{-1}f)=f^*\nu_0(f_nf^{-1})+\nu_0(f),
$$ 
and this converges to $\nu_0(f)$ since $f^*$ is an isometry. This concludes the proof.
\end{proof}

\begin{proof}[Proof of Proposition \ref{p:ell}] 

\noindent 1. The Taylor integral formula gives:
$$
\forall x\neq y\in[0,1],\quad \frac{f(x)-f(y)}{x-y}=\int_0^1 Df((1-t)x+ty)dt.
$$
If $f$ is $C^3$, the map $(x,y)\mapsto \frac{f(x)-f(y)}{x-y}$ thus extends to a (non-vanishing) $C^2$ map $\tau$ on $[0,1]^2$, 
so $\ell(f):(x,y)\mapsto\partial_x\partial_y\log\left(\frac{f(x)-f(y)}{x-y}\right)$ extends to a continuous map on $[0,1]^2$. More precisely, 
\begin{equation}
\label{e:partialtau}
\partial_x\partial_y\log\tau= \partial_x\left(\frac{\partial_y\tau}{\tau}\right)=\frac{\partial_x\partial_y\tau}{\tau}-\frac{(\partial_x\tau)(\partial_y\tau)}{\tau^2}.
\end{equation}
Now, by derivation under the integral,
$$
\partial_x\tau(x,y)=\int_0^1(1-t)D^2f((1-t)x+ty)dt
\quad\text{and}\quad
\partial_y\tau(x,y)=\int_0^1tD^2f((1-t)x+ty)dt,
$$
so the rightmost term in \eqref{e:partialtau} converges to 
$$
\left(\frac{D^2f(x_0)}{Df(x_0)}\right)^2\times \left(\int_0^1(1-t)dt\right)\times \left(\int_0^1tdt\right)=\frac14 Pf(x_0)
$$
when $(x,y)$ tends to $(x_0,x_0)$. Finally, 
$$
(\partial_x\partial_y\tau)(x,y)=\int_0^1t(1-t)D^3f((1-t)x+ty)dt,
$$
so the first term of the right hand side of \eqref{e:partialtau} converges to 
$$
\frac{D^3f(x_0)}{Df(x_0)}\times \left(\int_0^1t(1-t)dt\right)=\frac16 \frac{D^3f(x_0)}{Df(x_0)}
$$
when $(x,y)$ tends to $(x_0,x_0)$. Therefore, $\ell(f)(x,y)$ tends to 
$$
\frac16 \frac{D^3f(x_0)}{Df(x_0)}-\frac14 Pf(x_0)=\frac16 \left( \frac{D^3f(x_0)}{Df(x_0)}-\frac32Pf(x_0) \right) =\frac16Sf(x_0),
$$ 
as announced.

\medskip

\noindent 2. Let $\alpha>0$. Let $f\in \Diff^{2+\alpha}_+([0,1])$ and $C_f>0$ be such that, for all $x,y\in[0,1]$, 
$$
|D^2f(x)-D^2f(y)|\le C_f |x-y|^\alpha.
$$
For all $x,y\in[0,1]$, with $x\leq y$,
$$
\ell(f)(x,y) = \frac{1}{(f(x)-f(y))^2}\left(Df(x)Df(y)-\left(\frac{f(x)-f(y)}{x-y}\right)^2\right).
$$
By Taylor's integral formula,
\begin{align*}
Df(y)&=Df(x)+(y-x)D^2f(x)+(y-x)\int_0^1(D^2f(x+t(y-x))-D^2f(x))dt\\
&=Df(x)+(y-x)D^2f(x)+(y-x)\eps(x,y)\qquad\text{with $\eps(x,y)\le C_f(y-x)^\alpha$}\\
&=Df(x)(1+(y-x)Pf(x)+(y-x)\eps'(x,y))\qquad\text{with $\eps'(x,y)= \frac{\eps(x,y)}{Df(x)}$}.
\end{align*}
Furthermore,
\begin{align*}
\frac{f(x)-f(y)}{x-y} &= Df(x)+\tfrac12(y-x)D^2f(x)+(y-x)\int_0^1(1-t)(D^2f(x+t(y-x))-D^2f(x))dt\\
&=Df(x)+\tfrac12(y-x)D^2f(x)+(y-x)\delta(x,y)\quad\text{with $\delta(x,y)\le C_f(y-x)^\alpha$}\\
&=Df(x)\Big(1+\tfrac12(y-x)Pf(x)+(y-x)\delta'(x,y)\Big)\quad\text{with $\delta'(x,y)=\frac{\delta(x,y)}{Df(x)}$}.
\end{align*}
Hence 
\begin{small}
\begin{align*}
\ell(f)(x,y)&=\tfrac{\left(  Df(x)\right)^2}{(f(x)-f(y))^2}\left(\! \Big(1+(y\!-\!x)Pf(x)+(y\!-\!x)\eps'(x,y)\Big) \!-\! \Big(1+\tfrac12(y\!-\!x)Pf(x)+(y\!-\!x)\delta'(x,y)\Big)^{\!\!2} \! \right)\\
&=\tfrac{\left(Df(x)\right)^2}{(f(x)-f(y))^2}(y-x)(\eps'(x,y)-2\delta'(x,y)-(y-x)^2(\delta'(x,y))^2),
\end{align*}
and thus
\begin{align*}
|\ell(f)(x,y)|&\le C'_f\times \tfrac{1}{|f(x)-f(y)|^2}|y-x|^{1+\alpha}=C'_f\times\left|\tfrac{x-y}{f(x)-f(y)}\right|^2\times\tfrac1{|x-y|^{1-\alpha}} 
\le C''_f\times \tfrac1{|x-y|^{1-\alpha}},
\end{align*}
\end{small}where $C'_f$ and $C''_f$ are constants going to $0$ when $f$ goes to $\id$ in $C^{2+\alpha}$ topology. This shows both that $\ell(f)$ is $L^1$ and that $\|\ell(f)\|_{L^1}$ goes to $0$ when $f$ goes to $\id$ in $C^{2+\alpha}$ topology.

As in the previous proposition, this only proves that the map $\ell$ from $\Diff^{2+\alpha}_+([0,1])$ to $L^1 ([0,1]^2)$ is continuous at the origin. However, continuity at every point can be deduced from this by an extra manipulation of the cocycle relation. We leave the details to the reader.
\end{proof}

Thanks to the cocycle relations \eqref{e:coc-tau}, \eqref{e:coc-c} and \eqref{e:coc-P} and to Propositions \ref{p:tau} and \ref{p:ell}, the cocycle $\ell:G=\Diff^{>2}_+([0,1])\to\mathbb{B}=L^1([0,1]^2)$ and any linear combination of $\nu_0$, $\nu_1$ and $P$ from $G=\Diff^2_+([0,1])$ to $\mathbb{B}=L^1([0,1])$ fall right into the scope of \S\ref{ss:gen}. As we will see now, we understand well the drifts of $\nu_0$, $\nu_1$ and $P-\nu_0-\nu_1$. Things are not as clear for $\ell$, but in Section~\ref{s:necessary} we will determine at which diffeomorphisms its drift vanishes. 

\subsection{Interpretation of the drifts of $\nu_0$ and $\nu_1$}
\label{ss:drift-nu}

\begin{lem}
\label{l:T0}
For every $f\in\Diff^{>1}_+[0,1]$, for $i \in [\![ 0,1 ]\!] $,
$$\drift_{\nu_i}(f)=|\log Df(i)|.$$
\end{lem}

\begin{proof} 
The cases $i=0$ and $i=1$ being analogous, we will only give the proof in the former one. 
Let $f\in\Diff^{>1}_+[0,1]$. Let us first prove the inequality $\drift_{\nu_0}(f)\le|\log Df(0)|$. Given \eqref{e:gen-fact}, it suffices to find a sequence $u_n$ in $L^1$ such that $\|\nu_0(f)-(f^*u_n-u_n)\|_{L^1}\to|\log Df(0)|$. Let us show that this holds for $u_n=\frac1x \mathbb{1}_{[\frac1n,1]}$. Let $n\in\N^*$, and let us first assume that $f(\frac1n)<\frac1n$. 
Then, for all $x\in(0,1]$,
\begin{align*}
(\nu_0(f)-(f^*u_n-u_n))(x)
&=\frac{Df(x)}{f(x)}-\frac1x - \frac{Df(x)}{f(x)}\mathbb{1}_{[\frac1n,1]}(f(x))+\frac1{x}\mathbb{1}_{[\frac1n,1]}(x)\\
&=\left\{
\begin{matrix}
\nu_0(f)(x)& \!\! \text{if } x\in[0,\frac1n],\\
\frac{Df(x)}{f(x)}& \, \qquad \text{if }  x\in[\frac1n,f^{-1}(\frac1n)],\\
0& \qquad \text{if } x\in[f^{-1}(\frac1n),1].
\end{matrix}
\right.
\end{align*}
Hence,
\begin{equation}
\label{e:first}
\|\nu_0(f)-(f^*u_n-u_n)\|_{L^1}
=\int_0^{\frac1n}|\nu_0(f)|+\int_{\frac1n}^{f^{-1}(\frac1n)}\frac{Df(x)}{f(x)}dx 
= \int_0^{\frac1n}|\nu_0(f)|-\log nf(\tfrac1n).
\end{equation}
Similarly, if $f(\frac1n)>\frac1n$, one gets 
\begin{equation}
\label{e:second}
\|\nu_0(f)-(f^*u_n-u_n)\|_{L^1}
=\int_0^{f^{-1}(\frac1n)}|\nu_0(f)|+\int_{f^{-1}(\frac1n)}^{\frac1n}\frac{dx}x 
= \int_0^{f^{-1}(\frac1n)}|\nu_0(f)|-\log nf^{-1}(\tfrac1n).
\end{equation}
In both cases, the integral on the right hand side goes to $0$ when $n$ goes to the infinity since $\nu_0(f)$ is an $L^1$ function, and the following term goes to $\pm\log Df(0)$. If $\log Df(0)=0$, this shows the desired convergence. If $\log Df(0)<0$ (resp. $>0$), $f(\frac1n)<\frac1n$ (resp. $f(\frac1n)>\frac1n$) for all sufficiently small $n$, and \eqref{e:first} (resp. \eqref{e:second}) implies that $\|\nu_0(f)-(f^*u_n-u_n)\|_{L^1}$ goes to $-\log Df(0)$ (resp. $+\log Df(0)$) as claimed.  

We only need to prove the other inequality in the case $\lambda=\log Df(0)<0$ (the case $\lambda = 0$ is immediate, and the case $\lambda > 0$ follows directly by taking the inverse). By (the generalization of) Sternberg's linearization theorem for hyperbolic $C^{>1}$ germs of $(\R,0)$, one can conjugate $f$ by a $C^{>1}$ diffeomorphism to a diffeomorphism which coincides with the homothety of ratio $e^\lambda$ on $[0,\frac12]$ (for example). Since $\drift_{\nu_0}$ is invariant under conjugacy in $\Diff^{>1}_+([0,1])$, we can assume henceforth that $f$ itself has this 
property. Now fix $n\in\N^*$. For all $k\in[\![1,n]\!]$, on $f^{-k}([f(\frac12),\frac12])$, we have 
$$
\frac{Df^n}{f^n}=\frac{(Df^{n-k}\circ f^k)Df^k}{f^{n-k}\circ f^k}=\frac{e^{(n-k)\lambda}Df^k}{ e^{(n-k)\lambda}f^k}=\frac{Df^k}{f^k}.
$$
Hence 
\begin{align*}
\frac{\|\nu_0(f^n)\|_{L^1}}n
&\ge \frac1n\sum_{k=1}^{n}\int_{f^{-k+1}(\frac12)}^{f^{-k}(\frac12)}\left|\frac{Df^n(x)}{f^n(x)}-\frac1x\right|dx\\
&=\frac1n\sum_{k=1}^{n}\int_{f^{-k+1}(\frac12)}^{f^{-k}(\frac12)}\left|\frac{Df^k(x)}{f^k(x)}-\frac1x\right|dx
\ge \frac1n\sum_{k=1}^{n}\left|\int_{f^{-k+1}(\frac12)}^{f^{-k}(\frac12)}\left(\frac{Df^k(x)}{f^k(x)}-\frac1x\right)dx\right|.
\end{align*}
In view of Cesaró's theorem, it is thus sufficient to prove that the integrals in the last term go to $\pm\lambda$ when $k$ goes to infinity. And indeed,
\begin{align*}
\int_{f^{-k+1}(\frac12)}^{f^{-k}(\frac12)}\left(\frac{Df^k(x)}{f^k(x)}-\frac1x\right)dx
= \Big[\log f^k(x) - \log x\Big]_{f^{-k+1}(\frac12)}^{f^{-k}(\frac12)}
= \underbrace{\log\left(\frac{\frac12}{f(\frac12)}\right)}_{=-\lambda}-\underbrace{\log\left(\frac{f^{-k}(\frac12)}{f^{-k+1}(\frac12)}\right)}_{\to\log(1)=0},
\end{align*}
thus closing the proof. 
\end{proof}

\begin{cor}
\label{l:nu}
For all $f\in\Diff^{>1}_+[0,1]$ without interior fixed point, 
$$\drift_{\nu_0-\nu_1}(f)\ge\left|\log (Df(0)Df(1))\right|.$$
\end{cor}

\begin{proof} 
Let $\nu=\nu_0-\nu_1$. Without loss of generality, we can assume $f(x)>x$ on $(0,1)$, so that, necessarily, $Df(0)\ge1\ge Df(1)$. Then 
$$
\drift_\nu(f)\ge |\drift_{\nu_0}(f)-\drift_{\nu_1}(f)| = \Big||\log Df(0)|-|\log Df(1)|\Big| = \Big|\log(Df(0)Df(1))\Big|,
$$
as we wanted to prove.
\end{proof}

\subsection{A short digression: generating vector fields and Mather invariant}
\label{ss:mather}

Here, we briefly recall some classical facts about interval diffeomorphisms. We refer the reader to \cite{Yoc,Na11} for a complete account on the subject and for the original references such as \cite{Ko,Sz58,Ta}, etc. First, if $f$ is a $C^r$ diffeomorphism of $[0,1)$ fixing only $0$, with $r\ge2$, then $f$ is the time-$1$ map of a unique $C^1$ vector field on $[0,1)$, which we will call its \emph{generating vector field}, and which is in addition $C^{r-1}$ on the interior $(0,1)$, and even on the whole $[0,1)$ if $f-\id$ is not $C^r$-flat at $0$. This vector field $X$ satisfies in particular 
\begin{equation}
\label{e:invariance}
X\circ f^n=X\cdot Df^n\quad \text{(invariance under its flow)},
\end{equation}
\begin{equation}
\label{e:derivX}
DX(0)=\log Df(0), 
\end{equation}
and
\begin{equation}
\label{e:X0}
\lim_{x\to0}\frac{f(x)-x}{X(x)} =
\left\{
\begin{matrix}
1 & \text{ \, \, if } Df(0) = 1,\\
\frac{Df(0)-1}{\log Df(0)} & \text{otherwise.}
\end{matrix}
\right.
\end{equation}
We will also need the following direct consequence of \eqref{e:invariance}:
\begin{equation}
\label{e:clef}
\log Df^n=\log X\circ f^n-\log X, \quad\text{ and thus } \quad P(f^n)=(f^n)^*D\log X - D\log X.
\end{equation}

Now if $f$ is a $C^r$ diffeomorphism of $[0,1]$ without interior fixed point, it has a generating vector field $X$ on $[0,1)$ and a generating field $Y$ on $(0,1]$. Generically, these two vector fields do not coincide, and $f$ thus does not embed in a global $C^1$ flow. If $(\phi_X^t)_t$ and $(\phi_Y^t)_t$ denote the flows of these two vector fields, given a point $p\in(0,1)$, the maps $\psi_X:t\mapsto \phi_X^t(p)$ and $\psi_Y:t\mapsto \phi_Y^t(p)$ define two $C^r$ diffeomorphisms from $\R$ to $(0,1)$ (their derivatives $t\mapsto X(\phi_X^t(p))$ and $t\mapsto Y(\phi_X^t(p))$ are $C^{r-1}$) which both conjugate $f$ to the unit translation. The transfer map 
$$
M_f^{p,p}=\psi_Y^{-1}\circ \psi_X
$$
 is thus a $C^r$ diffeomorphism of $\R$ (fixing $0$) which commutes with the unit translation, and one easily sees that its derivative equals $\tfrac{X}{Y}\circ\psi_X$, and that a different choice of $p$ results in pre- and post-composing $M_f^{p,p}$ by translations. The class of the induced circle diffeomorphism, up to pre- and post-composition by rotations, is called the \emph{Mather invariant of $f$} and denoted by $M_f$. Hence, $f$ embeds in a global $C^1$ flow if and only if $M_f$ is the class of the rotations.  Note that the total variation $\var\log DM$ of $\log DM$ on the circle for some representative $M$ of $M_f$ does not depend on the choice of representative, so we will abusively denote it by $\var\log DM_f$. One thus has, for any $p\in(0,1)$,
\begin{equation}
\label{e:varlogDM}
\var\log DM_f 
= \var(\log D(\psi_Y^{-1}\circ \psi_X),[0,1])
= \var \big( \log\tfrac{X}{Y},[p,f(p)] \big) 
= \int_p^{f(p)}|\tfrac{DX}X-\tfrac{DY}Y|.
\end{equation}

In the next sections, we will also be interested in the diffeomorphism 
$$
m^{p,p}_f = \psi_X\circ \psi_Y^{-1}\in\Diff^r_+([0,1])
$$
which simply equals $\psi_Y M_f^{p,p}\psi_Y^{-1}$, and commutes with $f$. More generally, it conjugates the time-$t$ map of $Y$ to that of $X$ (recall that the time-$1$ maps are both equal to $f$).
 
As was already observed by Mather himself in \cite{Mather} (and later used with variations in \cite{BCVW,EBN,EBM25}), the Mather invariant of a diffeomorphism can be made trivial by a localized perturbation in the following sense. This will be decisive in Section \ref{s:drift}. 

\begin{lem}
\label{l:cancel}
Let $f\in\Diff^r_+([0,1])$ be such that $f(x)>x$ for all $x\in(0,1)$, let $p\in(0,1)$ and define $X$, $Y$, $\psi_X$, $\psi_Y$ as above. Let 
\begin{itemize}
\item $\phi_0$ and $\phi_{1/2}$ be $C^r$ circle diffeomorphisms coinciding with the identity near $0$ and near $1/2$ respectively, and such that $\phi=\phi_0\phi_{1/2}$ is the circle diffeomorphism induced by $(M_f^{p,p})^{-1}$ (hence $\phi_{1/2}$ also fixes $0$);
\item $\Phi_0$ (resp. $\Phi_{1/2}$) be the diffeomorphism of the real line supported in $[-1,0]$ (resp. $[-\frac52,-\frac32]$) and coinciding 
therein with the lift of $\phi_0$ (resp. $\phi_{1/2}$) fixing the endpoints;
\item $h_i$ be the diffeomorphism $\psi_{Y}\Phi_i\psi_Y^{-1}$, for $i=0$ and $\frac12$, supported in $I=[f^{-1}(p),p]$ and $\phi_Y^{-3/2}(I)$ respectively (which both fix the orbit of $p$ under $f$);
\item $h=h_0h_{1/2}=h_{1/2}h_0$.
\end{itemize}
Then $g=fh$ is a $C^r$-diffeomorphism of $[0,1]$ without interior fixed point, which coincides with $f$ on $\{f^n(p),n\in\Z\}$ and outside $[f^{-3}(p),p]$ and which has a trivial Mather invariant. Furthermore, on $[f^{-3}(p),f^{-2}(p)]$, 
$$g^{-3}f^3 =\psi_{Y}M_f^{p,p}\psi_Y^{-1}=m^{p,p}_f.$$
\end{lem}

\subsection{Interpretation of the drift of $\mu=P-\nu_0-\nu_1$}
\label{ss:drift-mu}

We can now prove Theorem \ref{t:mu}, which we recall:

\begin{thmB}
Let $\mu=P-\nu_0-\nu_1$. The equality  $\drift_\mu(f)=\var\log DM_f$ holds for every $f\in \Diff^2_+([0,1])$ without interior fixed point. 
\end{thmB}

\begin{proof}
This is an adaptation of \cite{EN21}, p. 30 and 37. In what follows, $X$ and $Y$ denote the left and right Szekeres vector fields of $f$, respectively. Without loss of generality, we assume $f(x)>x$ on $(0,1)$.\medskip

\noindent\textbf{First inequality.} Let us first prove that $\var\log DM_f\le \drift_\mu(f)$. According to \eqref{e:varlogDM}, for any $a\in(0,1)$,
\begin{equation}
\label{e:varlog}
\var\log DM_f 
= \int_a^{f(a)}\left|\tfrac{DX}X-\tfrac{DY}Y\right|
=\frac1{2n}\int_{f^{-n}(a)}^{f^n(a)}\left|\tfrac{DX}X-\tfrac{DY}Y\right|.
\end{equation}
Now 
\begin{align*}
\tfrac{DX}X-\tfrac{DY}Y 
& = (f^{-2n})^*\tfrac{DX}X - P(f^{-2n}) - \left((f^{2n})^*\tfrac{DY}Y - P(f^{2n})\right)\\
&=(f^{-2n})^*\tfrac{DX}X - \mu(f^{-2n})-\nu_0(f^{-2n})-\nu_1(f^{-2n}) - (f^{2n})^*\tfrac{DY}Y + \mu(f^{2n})+\nu_0(f^{2n})+\nu_1(f^{2n})\\
&=(f^{-2n})^*\left(\tfrac{DX}X-\tfrac1x-\tfrac1{x-1}\right)-(f^{2n})^*\left(\tfrac{DY}Y-\tfrac1x-\tfrac1{x-1}\right)+(f^{-2n})^*\mu(f^{2n})+\mu(f^{2n}).
\end{align*}
Thus the last integral in \eqref{e:varlog} is bounded above by the sum of 
$$
I_n=\int_{f^{-n}(a)}^{f^n(a)}\left|(f^{-2n})^*\left(\tfrac{DX}X-\tfrac1x-\tfrac1{x-1}\right)\right|, 
\quad 
J_n=\int_{f^{-n}(a)}^{f^n(a)}\left|(f^{2n})^*\left(\tfrac{DY}Y-\tfrac1x-\tfrac1{x-1}\right)\right|
$$
$$
K_n = \int_{f^{-n}(a)}^{f^n(a)}|(f^{-2n})^*\mu(f^{2n})|
\quad\text{and}\quad 
L_n=\int_{f^{-n}(a)}^{f^n(a)}|\mu(f^{2n})|.
$$
Now by a change of variables,
$$
K_n + L_n =  \int_{f^{-3n}(a)}^{f^{-n}(a)}|\mu(f^{2n})|+\int_{f^{-n}(a)}^{f^n(a)}|\mu(f^{2n})|\le \|\mu(f^{2n})\|_{L^1} 
= 2n \, \mathrm{drift}_{\mu} (f) + o(n).
$$
We are thus left with proving that $\frac1{2n}(I_n+J_n)$ goes to $0$ as $n$ goes to infinity. Regarding $I_n$,
$$
I_n=\int_{f^{-3n}(a)}^{f^{-n}(a)}\left|\tfrac{DX}X-\tfrac1x-\tfrac1{x-1}\right|\
le
\int_{f^{-3n}(a)}^{f^{-n}(a)}\left|\tfrac{DX}X-\tfrac1x\right|+\int_{f^{-3n}(a)}^{f^{-n}(a)}\left|\tfrac1{x-1}\right|.
$$
The last integral goes to $0$ when $n$ goes to infinity for the interval of integration goes to $\{0\}$ and $x\mapsto \frac1{x-1}$ is continuous near $0$. For the previous integral, let us observe that near $0$, one has $\frac{DX}X-\frac1\id = o(\frac1X)$. Indeed, in the hyperbolic case, that is if $DX(0)=\lambda\neq 0$, 
$$
\frac{DX(x)}{X(x)}-\frac1x 
= \frac{\lambda+o(1)}{\lambda x +o(x)}-\frac1x
=o\left(\frac1x\right)=o\left(\frac1{X(x)}\right)\quad [x\to0].
$$

In the parabolic case, near $0$, $DX(x)=o(1)$ and $X(x)=o(x)$, so $\frac1x=o(\frac1X)$, so again one has $\frac{DX}X-\frac1\id = o(\frac1X)$. Thus,
$$
\int_{f^{-3n}(a)}^{f^{-n}(a)}\left|\tfrac{DX}X-\tfrac1x\right|
=o\left( \int_{f^{-3n}(a)}^{f^{-n}(a)}\frac1X\right)=o(2n),
$$
hence $\frac{I_n}{2n} \to 0$, as required.
The case of $J_n$ is analogous, and we leave it to the reader.
\medskip

\noindent \textbf{Second inequality.}
According to \eqref{e:gen-fact}, in order to prove that $\drift_\mu(f)\le\var\log DM_f$, it suffices to find a sequence of $L^1$ functions $u_n$ such that $\|\mu(f)-(f^*u_n-u_n)\|_{L^1}$ goes to $\var\log DM_f$. Let us show that this holds for 
$$
u_n := \left\{
      \begin{array}{rcl}
          0 & \mbox{on} & [0,1]\setminus [y_n,z_n], \\
         \tfrac{DX}X-\tfrac1\id-\tfrac1{\id-1} & \mbox{on} & [y_n, f(a)], \\ 
         \tfrac{DY}Y-\tfrac1\id-\tfrac1{\id-1}  & \mbox{on} & [f(a),z_n], 
        \end{array} 
        \right .
$$
where $y_n$ (resp. $z_n$) is a sequence of points in $(0,1)$ converging to $0$ (resp. $1$) and $a$ is a point in $(0,1)$. We next compute the $L^1$-norm of the coboundary defect $\delta_n:=\mu(f)-(f^*u_n-u_n)$ which, as can be easily checked, coincides with
$$\left \{
      \begin{array}{rcl}
          0 & \mbox{on} & K_n=[ y_n, a] \cup [f(a), f^{-1}(z_n)], \\
         -\tfrac{DX}X+\frac1\id+\frac1{\id-1}  & \mbox{on} & I_n^-=[f^{-1}(y_n), y_n], \\ 
         f^*(\tfrac{DY}Y-\frac1\id-\frac1{\id-1}) & \mbox{on} & I_n^+=[f^{-1} (z_n), z_n], \\ 
         P(f) - \big(  f^*(\tfrac{DY}Y)- \tfrac{DX}X\big) & \mbox{on} & I=[a, f(a)], \\
         \mu(f) & \mbox{on} & B_n=(0,f^{-1}(y_n))\cup (z_n,1).
      \end{array}
   \right .$$
In particular, $\int_{K_n\cup B_n}|\delta_n|$ goes to $0$, since $\mu(f)$ is an $L^1$ function. Handling $\int_{I_n^{\pm}}|\delta_n|$ is very similar to what we've done in the proof of the first inequality. Indeed, 
$$
\int_{I_n^-}|\delta_n|
\le \int_{I_n^-}|\tfrac{DX}X-\tfrac1\id|+\int_{I_n^-}\frac1{1-\id}
=o\underbrace{\left(\int_{I_n^-}\frac1X\right)}_{=1}+o(1),
$$
and similarly for $\int_{I_n^{+}}|\delta_n|$. We are thus left with evaluating $\int_I|\delta_n|$. For this, let us observe that $P(f) = f^*(\tfrac{DY}Y)-\tfrac{DY}Y$. Therefore, on $I$, one simply has $\delta_n=\tfrac{DX}X-\tfrac{DY}Y$, and we have already seen that the integral of the absolute value of this expression on a fundamental interval is precisely $\var \log DM_f$. This completes the proof.
\end{proof}


\section{General properties of drifts of cocycles for isometric actions on $L^1$ spaces}
\label{s:drift}

\subsection{A localization formula}

Given $f\in\Diff_+^r([0,1])$ such that $f(x)>x$ for all $x \in (0,1)$, we will call {\em simple fundamental domain} of $F=f^{\otimes d}$ any set $\cD\subset[0,1]^d$ of the form $[0,f(p)]^d\setminus[0,p]^d$ with $p\in(0,1)$.

\begin{prop}[Localization]
\label{p:local}
Let $r\in[1,+\infty)$, $d\in\N^*$ and $c:G=\Diff_+^r([0,1])\to \mathbb{B}=L^1([0,1]^d)$ be such that $c(f\circ g)=(g^{\otimes d})^*c(f)+c(g)$ (i.e. $c$ is a cocycle for the natural isometric action of $G$ on $\mathbb{B}$). Then for every $f\in\Diff_+^r([0,1])$ such that $f(x)>x$ on $(0,1)$ and every simple fundamental domain $\cD\subset [0,1]^d$ of $F=f^{\otimes d}$, 
\begin{equation}
\label{e:local}
\drift_c(f)=\lim_{N\to+\infty}\|c(f^{2N})\|_{L^1(F^{-N}(\cD))}
\end{equation}
\end{prop}

The proof is a straightforward generalization of the proof of Proposition 5.1 in \cite{EN21}. It relies on the following lemma.

\begin{lem}
\label{l:local}
Under the assumptions and notations of Proposition \ref{p:local}, if 
$n>2N$, 
the expression 
\begin{equation*}
\label{e:local2}
\left| \;\|c(f^n)\|_{L^1([0,1]^d)} - (n-2N) \|c(f^{2N})\|_{L^1(F^{-N}(\cD))}\;\right|
\end{equation*}
is bounded from above by 
\begin{equation*}
\label{e:local3}
n \times \;\|c(f)\|_{L^1(\cup_{k\notin[-N,N[}F^{k}(\cD))} + (4N-1)\times \|c(f)\|_{L^1(\cup_{k\in[-N,N[}F^{k}(\cD))}.
\end{equation*}
\end{lem}

\medskip

\begin{proof}[Proof of Proposition \ref{p:local} from Lemma \ref{l:local}.] 
If we divide by $n$ each side of the inequality of the preceding lemma and let $n$ go to infinity, we obtain
\begin{small}
$$
\big| \drift_\ell (f) - \|c(f^{2N})\|_{L^1(F^{-N}(\cD))} \big| 
\leq \|c(f)\|_{L^1(\cup_{k\notin[-N,N[}F^{k}(\cD))}.
$$
\end{small}
Letting now $N$ go to infinity, the term of the right-hand side above converges to $0$, thus yielding the announced equality \eqref{e:local}. 
\end{proof}

\begin{proof}[Proof of Lemma \ref{l:local}]
For each $n \geq 1$ we have 
\begin{equation}
\label{suma-total}
\|c(f^n)\|_{L^1([0,1]^d)}  = \sum_{\ell = -\infty}^{+\infty} \|c(f^n)\|_{L^1(F^\ell(\cD))}  .
\end{equation}
There are five types of indices $\ell$ to analyze:  
\begin{enumerate}

\item If $\ell \leq -n - N$, then by the cocycle relation and changes of variable, we have 
$$  \|c(f^n)\|_{L^1(F^\ell(\cD))}  \leq \sum_{k=0}^{n-1}  \|c(f)\|_{L^1(F^{k+\ell}(\cD))} ,$$
and all the domains involved in the right-hand side, namely $F^{\ell}(\cD), \ldots, F^{\ell+n-1}(\cD)$,
are contained in $\cup_{k<-N}F^{k}(\cD)$, since $\, \ell + n   -1\leq -N-1 $. 

\item If $-n-N < \ell \le -n + N$, then $\|c(f^n)\|_{L^1(F^\ell(\cD))} $ is smaller than or equal to   
$$\sum_{k=0}^{-\ell - N -1} \|c(f)\|_{L^1(F^{k+\ell}(\cD))} 
+ \sum_{k = -\ell - N}^{n-1} \|c(f)\|_{L^1(F^{k+\ell}(\cD))} .$$
The domains involved in the first sum above are all contained in $\cup_{k<-N}F^{k}(\cD)$. Moreover, since $\ell + n - 1 \leq N-1$, the second sum is bounded from above by $\|c(f)\|_{L^1(\cup_{k\in[-N,N[}F^{k}(\cD))}$.
Note that this case arises for $2N$ possible values of $\ell$.

\item If $-n + N < \ell \leq -N$, then for $(x,y) \in F^\ell(\cD)$ we have $F^{-\ell-N}(x,y) \in F^{-N}(\cD)$ and $F^{-\ell+N}(x,y) \in F^N(\cD)$. Hence, applying the cocycle relation to $f^n = f^{n+\ell-N}\circ f^{2N}\circ f^{-\ell-N}$, we get 
$$
c(f^n) = (F^{-\ell+N})^*c(f^{n+\ell-N}) + (F^{-\ell-N})^*c(f^{2N})+c(f^{-\ell-N}).
$$
We conclude that the value of 
$$
\big| \|c(f^n)\|_{L^1(F^\ell(\cD))}  - \|c(f^{2N})\|_{L^1(F^{-N}(\cD))}  \big|
$$
is bounded from above by 
$$
\|c(f^{-\ell-N})\|_{L^1(F^\ell(\cD))} + \|c(f^{n+\ell-N})\|_{L^1(F^N(\cD))} .
$$
In its turn, this is bounded from above by 
$$
\sum_{k=0}^{-\ell - N - 1} \|c(f)\|_{L^1(F^{k+\ell}(\cD))}  + 
\sum_{k=0}^{ n + \ell - N - 1} \|c(f)\|_{L^1(F^{k+N}(\cD))} ,
$$
with the domains involved in the first (resp. second) sum being contained in $\cup_{k<-N}F^{k}(\cD)$ (resp. $\cup_{k\ge N}F^{k}(\cD)$). Notice that this case arises for $\, n \!-\! 2N \,$ values of $\ell$.

\item If $-N < \ell \leq N-1$, then $ \|c(f^n)\|_{L^1(F^\ell(\cD))}$ is bounded from above by 
$$
\sum_{k=0}^{N-\ell-1}  \|c(f)\|_{L^1(F^{k+\ell}(\cD))} 
+ \sum_{k=N-\ell}^{n-1} \|c(f)\|_{L^1(F^{k+\ell}(\cD))} ,
$$
which in its turn is smaller than or equal to 
$$
 \|c(f)\|_{L^1(\cup_{-N<k<N}F^{k}(\cD) )}+ 
\sum_{k=N-\ell}^{n-1}\|c(f)\|_{L^1(F^{k+\ell}(\cD))}  ,
$$
with the last sum involving only domains contained in $\cup_{k\ge N}F^{k}(\cD)$. Notice that this case arises for $\, 2N \! - \! 1 \,$ different values of $\ell$.

\item If $N \leq \ell$, then 
$$ 
\|c(f^n)\|_{L^1(F^\ell(\cD))} \leq \sum_{k=0}^{n-1} \|c(f)\|_{L^1(F^{k+\ell}(\cD))},
$$
and all the domains involved in the last sum are contained in $\cup_{k\ge N}F^{k}(\cD)$. 

\end{enumerate}

We come back to equality (\ref{suma-total}). Notice that the $L^1$ norm of $c(f)$ on each domain of the form $F^{\ell} (\cD)$ contained in $\cup_{k\notin[-N,N[}F^{k}(\cD)$ appears precisely $n$ times along the preceding estimates. Putting all of this together, one easily obtains the desired estimate just using the triangle inequality. 
\end{proof}

\subsection{Continuity and invariance under ``almost-conjugacy''}
\label{ss:continuity}

The localization formula from the previous paragraph allows us to prove an important continuity result 
for the drift. Again, this is a straightforward generalization of Theorem C in \cite{EN21}.

\begin{prop}[Continuity]
\label{c:continu}
Given $r\in[1,+\infty)$ and $d\in\N^*$, let $c:G=\Diff_+^r([0,1])\to \mathbb{B}=L^1([0,1]^d)$ be a continuous cocycle, that is, $c(f\circ g)=(g^{\otimes d})^*c(f)+c(g)$. Then the restriction of $\drift_c$ to $\Diff_+^{r,\Delta}([0,1])$ is continuous.
\end{prop}

Before giving the proof, let us state a direct corollary of this proposition and the invariance of the drift under conjugacy.

\begin{cor}[Invariance]
\label{c:invariance}
Under the assumptions of the previous statement, the restriction of $\drift_c$ to $\Diff_+^{r,\Delta}([0,1])$ is invariant under almost-conjugacy in $\Diff_+^{r}([0,1])$ in the following sense: given $f,g\in \Diff_+^{r,\Delta}([0,1])$, if there exists a sequence $(\varphi_n)$ in $\Diff_+^{r}([0,1])$ such that $\varphi_n g \varphi_n^{-1}$ converges to $f$ in $C^r$ topology, then $\drift_c(f)=\drift_c(g)$.
\end{cor}

\begin{proof}[Proof of Proposition \ref{c:continu}] 
In all that follows, when we simply write $\|\cdot\|$, we mean $\|\cdot\|_{L^1([0,1]^d)}$; otherwise we specify the domain of integration. Given $f \in \mathrm{Diff}_+^{r,\Delta} ([0,1])$, we want to prove the continuity of $\drift_c$ at $f$. Since $\drift_c$ is invariant under taking the inverse, we may assume that $f(x) \!>\! x$ for all $x \in (0,1)$. Fix $p \in (0,1)$. Lemma \ref{l:local} implies that, for $N\in\N$ and $n > 2N$, the value of 
\begin{equation}
\label{a-estimar-dist}
\left| \frac{\|c(f^n)\|}{n} -\|c(f^{2N})\|_{L^1(F^{-N}(\cD))} \right|
\end{equation}
is smaller than or equal to 
$$ 
\|c(f)\|_{L^1(\cup_{k\notin[-N,N[}F^{k}(\cD))}+ 
\tfrac{4N}{n} \|c(f)\| + \tfrac{2N}{n} \|c(f^{2N})\|.
$$
Given $\varepsilon > 0$, let $\delta > 0$ be such that $\|c(f)\|_{L^1([0,1]^d\setminus[\delta,1-\delta]^d)}< \frac{\varepsilon}{8}$. Let $\cD_f$ be the polygon $([0,f(p)]\times [p,f(p)])\cup ([p,f(p)]\times[0,f(p)])$, which is a fundamental domain for $F$.
Fix $N$ large enough so that $\cup_{k\notin[-N,N[}F^k(\cD_f)$ lies in the neighborhood $[0,1]^d\setminus[\delta,1-\delta]^d$ of the boundary of $[0,1]^d$. Next, consider any diffeomorphism $g \in \mathrm{Diff}_+^{r,\Delta} ([0,1])$ that is close enough to $f$ in the $C^{r}$ topology so that the following conditions are satisfied:
\begin{itemize}

\item $g(x) > x$ for all $x \in (0,1)$,

\item $\cup_{k\notin[-N,N[}G^k(\cD)\subset [0,1]^2\setminus[\delta,1-\delta]^2$, where $G=g^{\otimes d}$,

\item $\|c(g)\|_{L^1([0,1]^d\setminus[\delta,1-\delta]^d)}< \frac{\varepsilon}{8}$,

\item $\|c(g)\| < \|c(f)\|+1 \,$ and $\, \|c(g^N)\| < \|c(f^N)\| +1$, 

\item $ \big| \|c(f^{2N})\|_{L^1(F^{-N}(\cD_f))} - \|c(g^{2N})\|_{L^1(G^{-N}(\cD_g))}
\big| < \frac{\varepsilon}{4}$, where $\cD_g$ denotes the polygon 
$$([0,g(p)]\times [p,g(p)])\cup ([p,g(p)]\times[0,g(p)]).$$
\end{itemize}

\noindent Now consider any large enough integer $n$ such that $n > 2N$ and 
$$
\tfrac{4N}{n} \big( \|c(f)\|+ 1 \big) + \tfrac{2N}{n} \big( \|c(f^{2N})\| + 1 \big) < \tfrac{\varepsilon}{8}.
$$ 
Since the estimate given for expression (\ref{a-estimar-dist}) holds when replacing $f$ by $g$, we easily conclude from the previous conditions that 
$$
\left| \frac{\|c(f^n)\|}{n} - \frac{\|c(g^n)\|}{n}\right| < \varepsilon.
$$
Since this holds for any large enough $n$, passing to the limit we conclude that 
$$
\big| \drift_c(f)-\drift_c(g) \big| \leq \varepsilon.
$$
This shows the continuity of $\drift_c$ at $f$.
\end{proof}

\subsection{Sufficient condition for zero drift}
This section is devoted to the proof of Theorem \ref{t:vanish-drift-gen}, which we next recall:

\begin{thmA} 
Let $r\ge2$ and $c:\Diff^r_+([0,1])\to L^1([0,1]^d)$ be a continuous cocycle for the isometric action $f\in\Diff^r_+([0,1])\mapsto (f^{\otimes d})^*\in \mathrm{Isom}(L^1([0,1]^d))$. If $f\in\Diff^{r,\Delta}_+([0,1])$ has vanishing asymptotic variation, then $\drift_c(f)=0$.
\end{thmA}

The last missing ingredient to prove this is the following approximation result. We say for short that a diffeomorphism of the interval is \emph{$C^r$ parabolically flowable} if it is the time-$1$ map of a $C^r$ vector field without hyperbolic zero (a \emph{$C^r$ parabolic vector field}, for short).

\begin{prop}
\label{p:lissage} Let $r\ge2$. If $f\in\Diff^{r,\Delta}_+([0,1])$ has vanishing asymptotic variation (or equivalently: is ``$C^1$ parabolically flowable''), then there exists a sequence of $C^r$ parabolically flowable elements of $\Diff^{r,\Delta}_+([0,1])$ that $C^r$-converges to $f$.
\end{prop}

\begin{proof}[Proof of Theorem \ref{t:vanish-drift-gen}]
Assume $f\in\Diff^{r,\Delta}_+([0,1])$ is $C^1$ parabolically flowable, and let $(f_n)$ be a sequence given by Proposition \ref{p:lissage}. By the continuity property of Corollary \ref{c:continu}, one has $\drift_c(f)=\lim_{n\to+\infty}\drift_c(f_n)$. Now it was proved in \cite{EN21} that each $f_n$ is $C^r$-almost-reducible, and thus has vanishing $\drift_c$ according to \eqref{e:upper}, which concludes the proof.
\end{proof}

Proposition \ref{p:lissage} is proved in two steps. In the first one, we approach $f$ by diffeomorphisms that are $C^r$-flowable \emph{near the boundary} but not necessarily globally (Lemma \ref{l:pre-lissage} below). In the second one, we perform local perturbations on these diffeomorphisms to ``cancel their Mather invariant'' so that they embed in a global $C^r$ flow (the techniques for doing this were already employed in \cite{BCVW,EBN,EBM25}). 

\begin{lem}
\label{l:pre-lissage}
 Let $r\ge2$, $p\in(0,1)$ and $f\in\Diff^{r}_+([0,1])$ such that $f(x)>x$ for all $x\in(0,1)$. For every $\delta>0$, there exists $g\in\Diff^{r,\Delta}_+([0,1])$ such that $d_{C^r}(f,g)<\delta$, $g=f$ on $[f^{-3}(p),p]$ and $g$ locally embeds in a $C^r$ parabolic flow near the endpoints. 
 \end{lem}

\begin{proof}[Proof of Lemma \ref{l:pre-lissage}] 
This is very similar to the proof of Lemma 3.2 in \cite{EBN}. It is not difficult to construct a smooth vector field $Z$ on $[0,1]$ positive on $(0,1)$ and whose time-$1$ map $g$ has the same $r$-jets as $f$ at the endpoints. Next, fix a 
smooth even map $\rho:\R\to[0,1]$ equal to $1$ on $[-\frac13,\frac13]$, to $0$ outside $(-1,1)$ and positive on $(-1,1)$. For $a=0$ or $1$ and for every $\eps>0$, let $\rho_{a,\eps}$ be the function  $x\mapsto \rho(\frac{x-a}\eps)$. Let us now define $h_\eps$ by 
$$
h_\eps=
\left\{
\begin{matrix}
\rho_{0,\eps}g+(1-\rho_{0,\eps})f&\text{on}&\hspace{-0.7cm}[0,\eps],\\
f\hspace{3cm}&\text{on}& \, [\eps,1-\eps],\\
\rho_{1,\eps}g+(1-\rho_{1,\eps})f&\text{on}&\,\,[1-\eps,1].
\end{matrix}
\right.
$$
The fact that $h_\eps$ is $C^r$, fixes (only) $0$ and $1$ and is $C^r$ flowable near the endpoints is immediate. Let us check that it $C^r$-converges to $f$ when $\eps$ goes to $0$ (in particular, this will imply that $h_\eps$ has non-vanishing derivative and thus is a diffeomorphism for $\eps$ small). Note that $h_\eps-f$ is $C^r$-flat at the endpoints, hence by the mean value theorem (applied several times) one has 
$$
\|h_\eps-f\|_{r}\le \max_{x\in[0,1]}|D^r(h_\eps-f)(x)|.
$$ 
Let us thus evaluate this quantity, starting with $x\in [0,\eps]$. On this interval, $h_\eps-f=\rho_{0,\eps}(g-f)$, so the Leibniz rule gives 
\begin{equation}
D^r(h_\eps-f)=\sum_{k=0}^r \binom{r}{k} D^k\rho_{0,\eps}\times D^{r-k}(g-f), 
\end{equation}
and thus 
$$
\max_{[0,\eps]}|D^r(h_\eps-f)|\le \sum_{k=0}^r \binom{r}{k} \frac{\|D^k\rho\|_\infty}{\eps^k}\max_{[0,\eps]}|D^{r-k}(g-f)|.$$
But, again by (an iterated use of) the mean value theorem, 
$$
\max_{[0,\eps]}|D^{r-k}(g-f)|\le\max_{[0,\eps]}|D^r(g-f)| \eps^k=o(\eps^k),
$$
so we get that when $\eps$ goes to $0$, $\max_{[0,\eps]}|D^r(h_\eps-f)|$ goes to $0$. One proceeds in the same way on $[1-\eps,1]$, and, since $h_\eps=f$ on $[\eps,1-\eps]$, we indeed get that $h_\eps$ converges to $f$ in $C^r$ topology when $\eps$ goes to $0$. Finally, if we take $\eps$ small enough that $[f^{-3}(p),p]\subset [\eps,1-\eps]$, then $h_\eps$ satisfies the last requirement of the statement.
\end{proof}

What follows is a variation on arguments detailed in \cite{BCVW,EBN,EBM25}. Here, we provide a sketch of the proof, referring the reader to the above articles for more details. 

\begin{proof}[Proof of Proposition \ref{p:lissage}]
Let $f\in\Diff^{r,\Delta}_+([0,1])$ be $C^1$ parabolically flowable and $p\in(0,1)$. Lemma \ref{l:pre-lissage} provides a sequence of diffeomorphisms $f_n\in\Diff^{r,\Delta}_+([0,1])$ converging to $f$ in $C^r$ topology, coinciding with $f$ on $[f^{-3}(p),p]$ and $C^r$ flowable \emph{near the endpoints} but maybe not globally. In other words, the $f_n$ may have non-trivial Mather invariants, and we need to ``get rid of these'' by $C^r$-small perturbations. Lemma \ref{l:cancel} tell us exactly how to perform such perturbations with support in $[f^{-3}(p),p]$. The fact that they can be made $C^r$-small when $n$ is big comes from the work of Yoccoz: on the one hand, the Mather invariants of the $f_n$ converge to that of $f$, that is the identity, in $C^r$ topology, and on the other hand, if $Y_n$ denotes the right generating vector field of $f_n$, then $\psi_{Y_n}:[3,0]\to[f^{-3}(p),p]$ converges in $C^r$ topology to $\psi_{Y}:[3,0]\to[f^{-3}(p),p]$ (cf. \cite{Yoc}, chap. V, \S2.3 and chap. IV, \S 2.5 respectively). 
\end{proof}

\subsection{Expression for the drift in terms of the left and right flows}
\label{ss:formula}

We now provide a new formula for the drift at a parabolic diffeomorphism without interior fixed point, in terms of its left and right generating vector fields (the notations are those of \S\ref{ss:mather}), for a cocycle $c:\Diff^r_+([0,1])\to L^1([0,1])$. This corresponds to the case $d=1$ in our general setting, and we don't know whether an analogous formula holds for $d > 1$. 

\begin{thm} 
\label{t:formula}
Let $r\ge2$ and $c:\Diff^r_+((0,1))\to L^1_{\mathrm{loc}}((0,1))$ satisfy 
\begin{enumerate}
\item $c(f)_{|I}$ depends only on $f_{|I}$, 
\item $c(fg)=g^*c(f)+c(g)$, 
\item $c$ induces a continuous map from $\Diff^r_+([0,1])$ to $L^1([0,1])$.
\end{enumerate}
If $f\in\Diff^{r,\Delta}_+([0,1])$ satisfies $f(x)>x$ on $(0,1)$ and is parabolic at the endpoints, then for every $p\in(0,1)$, letting $\cD=[p,f(p)]$,
$$
\drift_c(f)=\lim_{n\to+\infty}\int_{f^{-n}(\cD)}|c(m_f^{p,p})|.
$$
\end{thm}

\begin{proof}
Let $f$, $p$ and $\cD$ be as in the statement, and let $g=fh$ be the diffeomorphism given by Lemma \ref{l:cancel}. Note that $f^{-n}(\cD)=g^{-n}(\cD)$ for every $n\in\N$. According to Theorem \ref{t:vanish-drift-gen} and Proposition~\ref{p:local}, 
$$
\drift_c(g)=\lim_{n\to\infty}\int_{g^{-n}(\cD)}|c(g^{2n})|=0.
$$
Given $n\in\N$, let $H_n=g^{-2n}f^{2n}$, so that $f^{2n}=g^{2n}H_n$ and $c(f^{2n})=H_n^*c(g^{2n})+c(H_n)$. Hence 
$|c(f^{2n})-c(H_n)|=|H_n^*c(g^{2n})|$ and in particular 
$$
\left|\int_{f^{-n}(\cD)}|c(f^{2n})|-\int_{f^{-n}(\cD)}|c(H_n)|\right|\le \int_{f^{-n}(\cD)}|H_n^*c(g^{2n})|=\int_{g^{-n}(\cD)}|c(g^{2n})|\xrightarrow[n\to\infty]{}0.
$$
Using Proposition~\ref{p:local} again, this implies 
$$
\drift_c(f)=\lim_{n\to+\infty}\int_{f^{-n}(\cD)}|c(H_n)|.
$$
Now on $f^{-n}(\cD)$,
\begin{align*}
H_n = g^{-2n}f^{2n}= g^{-n}(g^{-n}f^n)f^n = g^{-n}f^n
\end{align*}
because $f^n=g^n$ on $f^n(f^{-n}(\cD))=[p,f(p)]$ (and even on $[p,1]$). Furthermore, still on $f^{-n}(\cD)$, using the last sentence of Lemma \ref{l:cancel} and the fact that $f^{n-3}=g^{n-3}$ on $f^{-n}(\cD)$,
$$
g^{-n}f^n = g^{-n+3}(g^{-3}f^3)f^{n-3}=(f^{n-3})^{-1}m_f^{p,p}f^{n-3}.
$$
We conclude the proof by recalling that $m_f^{p,p}$ commutes with $f$.
\end{proof}

Let us conclude this section by applying this to the cocycle $P:f\mapsto D\log Df$ (cf. Introduction) and recovering the following result of \cite{EN21}: if $f\in\Diff^{2,\Delta}_+([0,1])$ is parabolic at the endpoints, then
$$
\drift_P(f)=\var\log DM_f.
$$
According to Theorem \ref{t:formula}, 
\begin{align*}
\drift_P(f)&=\lim_{n\to+\infty}\int_{f^{-n}(\cD)}|P(m_f^{p,p})|.
\end{align*}
Let us abbreviate $m^{p,p}_f$ by $m_f$. Recall that $m_f=\psi_X\psi_Y^{-1}$, hence 
$$
Dm_f=\frac{D\psi_Y^{-1}}{D\psi_X^{-1}(m_f)}=\frac{X\circ m_f}{Y}.
$$
Thus
$$
\log Dm_f = \log(X\circ m_f)-\log Y,
$$
and therefore 
$$
P(m_f)=\tfrac{DX\circ m_f - DY}Y=\left(\tfrac{DX}X-\tfrac{DY}Y\right)+\left(\tfrac{DX\circ m_f}Y-\tfrac{DX}X\right).
$$
Hence, recalling that $m_f$ fixes $f^{-n}(\cD)$,
\begin{align*}
\int_{f^{-n}(\cD)}|P(m_f)-\left(\tfrac{DX}X-\tfrac{DY}Y\right)|
&\le \max_{f^{-n}(\cD)}|DX|\left(\int_{f^{-n}(\cD)}\tfrac1Y+\int_{f^{-n}(\cD)}\tfrac1X\right).
\end{align*}
The last two integrals are equal to $1$ and $DX$ is continuous at $0$, where it vanishes. Therefore, we finally get 
$$
\lim_{n\to+\infty}\int_{f^{-n}(\cD)}|P(m_f)| = \lim_{n\to+\infty}\int_{f^{-n}(\cD)}|\tfrac{DX}X-\tfrac{DY}Y|=\var\log DM_f.
$$


\section{Criterion for the vanishing of the Liouville drift}
\label{s:necessary}

Here, we prove Theorem \ref{t:criterion}, which we recall:

\begin{thmC}
Given $f\in \Diff^{>2}_+([0,1])$, the Liouville drift vanishes at $f$ if and only if $f$ is $C^{>2}$-conjugated to the restriction to $[0,1]$ of a Möbius map of 
$\R\cup\{\infty\}$ or $f$ embeds in a $C^1$ flow without hyperbolic fixed points (or, equivalently, has vanishing asymptotic variation).
\end{thmC}

When there are no interior fixed points, the ``sufficient'' character of the condition comes from Theorem \ref{t:vanish-drift-gen}, the fact that $\ell$ vanishes at Möbius maps (cf. Lemma \ref{l:mobius} below) and that $\drift_\ell$ is a $C^{>2}$-conjugacy invariant. The necessary character is the content of Corollaries \ref{c:necessary} and \ref{c:multip} and comes from a relation between the cocycles $P,\nu_0,\nu_1$ and $\ell$ (cf. Lemma \ref{l:relation}). The general statement uses a fragmentation property established in the next subsection. In Section \ref{s:circle}, we deal with the case of the circle (for which the definition of $\ell$ must be slightly modified).

\begin{lem}
\label{l:mobius}
For $f\in\Diff^1_+([0,1])$, the equality $\ell(f)=0$ holds if and only if $f$ is the restriction to $[0,1]$ of a Möbius transformation of $\R\cup\{\infty\}$ fixing $0$ and $1$.
\end{lem}

\begin{proof} 
If $0\le a<b<c<d\le1$, a direct integration shows that
\begin{equation}
\label{eq:biraport} 
\int_a^b \int_c^d \frac{dx \, dy}{(x-y)^2} = \log \left( \frac{(d-b)(c-a)}{(d-a)(c-b)} \right) = \log ([a,b,c,d]),
\end{equation}
where $[w,x,y,z]$ denotes the cross-ratio of the points $w,x,y,z$ of $\R$.
It follows directly from the definition of $\ell(f)$ that 
$$
\iint_{[a,b]\times[c,d]} \ell(f) = \log \left(\frac{[f(a),f(b),f(c),f(d)]}{[a,b,c,d]}\right).
$$
Hence, if $\ell(f)=0$, $f$ must preserve the cross-ratios of quadruples satisfying $0 \!\le\! a\! < b\! < \!c\! < \!d\! \le \! 1$, and only a Möbius transformation can do this. Conversely, if $f$ is a Möbius transformation, then a straightforward computation shows that $\ell (f) = 0$ everywhere. (Alternatively, for every $x<y\in(0,1)$ and every small $\eps>0$, we get $\iint_{[x-\eps,x]\times[y,y+\eps]} \ell(f)=0$ and, by continuity, this forces $\ell(f)(x,y)$ to vanish.)
\end{proof}

\subsection{Fragmentation property}
\label{ss:fragment}

\begin{notation} 
Given $f\in\mathrm{Diff}^{>2}_+([0,1])$, for any compact interval $I$ stable under $f$, we let 
$$
\drift_{\ell,I}(f):=\lim_{n\to+\infty}\frac{\|\ell(f^n)\|_{L^1(I)}}n\le\|\ell(f)\|_{L^1(I^2)}.
$$
\end{notation}
For every collection $\cC$ of compact subintervals of $[0,1]$ stable under $f$ and with disjoint interiors, it is immediate that
\begin{equation}
\label{e:loc}
\drift_{\ell,[0,1]}(f)\ge \sum_{I\in\cC}\drift_{\ell,I}(f).
\end{equation}
Proposition \ref{p:fragment} claims that this is actually an equality if $\cup_{I\in\cC}I=[0,1]$, provided $f$ is parabolic at the endpoints of the intervals of $\cC$ that lie inside $(0,1)$. It will follow from Proposition \ref{p:multip}, however, that this fails to be true in the event of hyperbolic interior fixed points. 

\begin{prop}
\label{p:fragment} 
Let $f\in\Diff^{>2}_+([0,1])$ and let $\cC$ be a collection of compact subintervals of $[0,1]$ covering $[0,1]$, with non empty disjoint interiors, stable under $f$ and such that $f$ is parabolic at every endpoint of an element of $\cC$ that lies in $(0,1)$. Then
$$
\drift_{\ell,[0,1]}(f)=\sum_{I\in\cC}\drift_{\ell,I}(f).
$$
\end{prop}

The key ingredient is the following lemma:

\begin{lem}
\label{l:fragment}
If $f\in\Diff^{>2}_+([0,1])$ has $a\in(0,1)$ as a parabolic fixed point, then 
$$\drift_{\ell,[0,1]}(f)=\drift_{\ell,[0,a]}(f)+\drift_{\ell,[a,1]}(f).$$
\end{lem}

\begin{rem}
Upcoming results will imply that this cannot be true without the ``parabolic'' assumption. Indeed, assume for example that $f\in\Diff^{>2}_+([0,1])$ has a unique hyperbolic fixed point at $a$, that $Df(0)=\frac1{Df(a)}=Df(1)$ and that $f_{|[0,a]}$ and $f_{|[a,1]}$ embed in a $C^1$ flow. In this case, $f_{|[0,a]}$ and $f_{|[a,1]}$ are $C^{>2}$-conjugated to Möbius maps (cf. Corollary \ref{c:multip}), and thus, since the drift is invariant under conjugacy, 
$\drift_{\ell,[0,a]}(f) =  \drift_{\ell,[a,1]}(f)=0$. However, we will see in the next subsection that a diffeomorphism of $[0,1]$ with an interior hyperbolic fixed point cannot have vanishing Liouville drift. We will clearly see in the proof below where the parabolicity comes into play (cf. Remark \ref{r:hyp}).
\end{rem}

\begin{proof}[Proof of Lemma \ref{l:fragment}]
As we already pointed out, the inequality 
$$
\drift_{\ell,[0,1]}(f)\ge\drift_{\ell,[0,a]}(f)+\drift_{\ell,[a,1]}(f)
$$  
is immediate. Let us now show that for every $\eps>0$, 
$$
\drift_{\ell,[0,1]}(f)\le \drift_{\ell,[0,a]}(f)+\drift_{\ell,[a,1]}(f)+\eps.
$$ 
According to \eqref{eq-drift-cob}, it suffices to find $u\in L^1([0,1]^2)$ such that 
\begin{equation}
\label{e:ineg-cobord}
\|\ell(f)-(u-f^*u)\|_{L^1([0,1]^2)}\le  \drift_{\ell,[0,a]}(f)+\drift_{\ell,[a,1]}(f)+\eps.
\end{equation}
Let $I=[0,a]$ and $J=[a,1]$. Again, thanks to \eqref{eq-drift-cob} applied to $f_{|I}$ and $f_{|J}$, we know that there exists $v\in L^1(I^2\cup J^2)$ such that 
\begin{equation}
\label{e:cobord2}
\|\ell(f)-(v-f^*v)\|_{L^1(I^2\cup J^2)}\le \drift_{\ell,I}(f)+\drift_{\ell,J}(f)+\tfrac\eps2.
\end{equation} 
Let $\iota$ denote the reflexion $(x,y)\mapsto (y,x)$ and, given a subset $A$ of $[0,1]^2$, define $\sigma(A)$ as $A\cup \iota(A)$.
For $n\ge\max(\frac1a,\frac1{1-a})$, let us define $u_n\in L^1([0,1]^2)$ by 
\begin{itemize}
\item $u_n=v$ on $I^2\cup J^2$, 
\item $u_n=0$ on $\sigma(R_n)$, where $R_n=(a-\frac1n,a)\times (a,a+\frac1n)$,
\item $u_n(x,y)=-\frac1{(x-y)^2}$ everywhere else.
\end{itemize}
The coboundary defect $\delta_n:=|\ell(f)-(u_n-f^*u_n)|$ is then equal to
$$\left \{
      \begin{array}{rl}
         |\ell(f)-(v-f^*v)|  & \mbox{on} \; I^2\cup J^2, \\ 
         |\ell(f)| & \mbox{on} \; \sigma(R_n)\cap(f^{\otimes2})^{-1}(\sigma(R_n))=:A_n, \\
         |f^*u_n| & \mbox{on} \; \sigma(R_n)\setminus(f^{\otimes2})^{-1}(\sigma(R_n))=:B_n, \\
         |u_n| & \mbox{on} \; (f^{\otimes2})^{-1}(\sigma(R_n))\setminus \sigma(R_n)=:C_n, \\
         0 &\mbox{elsewhere}  
      \end{array}
   \right .$$
(recall that, if $u$ is the function $(x,y)\mapsto \frac1{(x-y)^2}$ on $[0,1]^2$ minus the diagonal, then $\ell(f)=f^*u-u$).
Thus, 
\begin{align*}
\label{e:cobord3}
\|\delta_n\|_{L^1([0,1]^2)}
&\le \int_{I^2\cup J^2}|\ell(f)-(v-f^*v)|+\int_{A_n}|\ell(f)|+\int_{B_n}|f^*u_n|+\int_{C_n}|u_n|\\
&\le \drift_{\ell,I}(f)+\drift_{\ell,J}(f)+ \frac\eps2+\int_{A_n}|\ell(f)|+\int_{f^{\otimes 2}(B_n)\cup C_n}|u_n|.
\end{align*}
Since $\ell(f)$ is $L^1$ on $[0,1]^2$ and the measure of $A_n$ goes to $0$ when $n$ goes to infinity, the term $\int_{A_n}|\ell(f)|$ is less than $\frac\eps4$ for $n$ large enough. Let us now evaluate $\int_{f^{\otimes 2} (B_n)}|u_n|$. Let $a_n^\pm~=~a~\pm~\frac1n$. Suppose that $f(a_n^-) \leq a_n^-$ and $f(a_n^+) \geq a_n^+$; the other situations can be handled similarly (note however that the direction of the inequalities might depend on $n$ since there could be an accumulation of fixed points at $a$, and the sign of $f-\id$ might change infinitely many times near $a$). 
In this situation, one easily checks that
$$
f^{\otimes 2}(B_n)
= \sigma(\underbrace{[f(a_n^-),a_n^-]\times[a,f(a_n^+)]}_{D_n}\cup\underbrace{[a_n^-,a]\times[a_n^+,f(a_n^+)]}_{E_n}).
$$
Using (\ref{eq:biraport}), we obtain 
$$
\int_{D_n}|u_n| = \int_{f(a_n^-)}^{a_n^-}  \int_{a}^{f(a_n^+)}  \frac{dx \, dy}{(x-y)^2} 
=\log \underbrace{\left(\frac{f(a_n^+)-a_n^-}{f(a_n^+)-f(a_n^-)}\right)}_{\le1} 
+ \log \underbrace{\left(\frac{a-f(a_n^-)}{a-a_n^-}\right)}_{\xrightarrow[n\to\infty]{}Df(a)=1} 
$$
and
$$
\int_{E_n}|u_n| =  \int_{a_n^-}^{a} \int_{a_n^+}^{f(a_n^+)}\frac{dx \, dy}{(x-y)^2} 
=\log \underbrace{\left(\frac{f(a_n^+)-a}{a_n^+-a}\right)}_{\xrightarrow[n\to\infty]{}Df(a)=1}
+\log \underbrace{\left(\frac{a_n^+-a_n^-}{f(a_n^+)-a_n^-}\right)}_{\le1}.
$$
Thus, if $n$ is large enough, $\int_{f^{\otimes 2}(B_n)}|u_n|<\frac\eps8$. Furthermore, in the situation we are considering, $C_n=\varnothing$. This shows that, in this case, if $n$ is large enough, then 
$$
\|\delta_n\|_{L^1([0,1]^2)}
\le \drift_{\ell,I}(f)+\drift_{\ell,J}(f)+ \eps.
$$

The case where $f(\alpha_n^-) \geq \alpha_n^-$ and $f(\alpha_n^+) \leq \alpha_n^+$ can be treated similarly, replacing $f$ by $f^{-1}$ (in this case, it is $B_n$ which is empty). Up to replacing $f$ by $f^{-1}$ again, there is one case left to study, namely $f(a_n^-) \leq a_n^-$ and $f(a_n^+) \leq a_n^+$. We leave to the reader to check that, again, $\int_{f^{\otimes 2}(B_n)\cup C_n}|u_n|$ can be estimated in terms of $\log Df(a)$, which vanishes. 

Hence for all $n$ large enough, we get that $\|\delta_n\|_{L^1([0,1]^2)}\le \drift_{\ell,I}(f)+\drift_{\ell,J}(f)+ \eps$, 
so \eqref{e:ineg-cobord} is satisfied for $u=u_n$. This concludes the proof.\end{proof}

\begin{rem}
\label{r:hyp}
In the proof above, if $f$ is hyperbolic at $a$ with $Df(a)=\lambda>1$, an easy computation shows that $\int_{D_n}|u_n|$ goes to $\log(\lambda+1)-\log(2)>0$ and $\int_{E_n}|u_n|$ goes to $\log(\lambda)-\log(2)+\log(\lambda+1)>0$, so the proof breaks in this case, as announced.
\end{rem}

\begin{proof}[Proof of Proposition \ref{p:fragment}] 
A straightforward induction yields the case where the collection $\cC$ is finite. Let us now assume it is infinite. It suffices to prove that for every $\eps>0$,
\begin{equation}
\label{e:frag-eps}
\drift_{\ell,[0,1]}(f)\le\sum_{I\in\cC}\drift_{\ell,I}(f)+\eps.
\end{equation}
Fix $\eps>0$, and assume $f$ is $C^{2+\alpha}$ for some positive $\alpha$. By Proposition \ref{p:ell} and the absolute continuity of the integral, there exists $\delta > 0$ such that, if $A \subset [0,1]^2$ is any subset of measure less than $\delta$, then 
\begin{equation}
\label{eq:abs-cont}
\int_{A} | \ell (f) | \leq \varepsilon.
\end{equation}
Let $A$ be a small open neighborhood of the diagonal $\{(x,x) \!: x \in [0,1]\} \subset [0,1]^2$, of measure less than $\delta$. It is easy to see that one can choose finitely many elements $I_1,\dots,I_k$ of $\cC$ such that, if $I_{k+1},\dots,I_m$ are the connected components of $[0,1]\setminus \cup_{1\le i\le k} I_i$, then $I_j^2  \subset D$ for all $j \geq k+1$. Then, according to \eqref{eq:abs-cont},  
\begin{align*}
\drift_{\ell,[0,1]}(f)=\sum_{i=1}^m\drift_{\ell,I_i}(f)
&\le \sum_{i=1}^k\drift_{\ell,I_i}(f) + \sum_{j=k+1}^m\|\ell(f)\|_{L^1(I_j^2)}\\
&\le \sum_{i=1}^k\drift_{\ell,I_i}(f) + \int_{A} | \ell(f) |  \le \sum_{I\in\cC}\drift_{\ell,I}(f) + \eps, 
\end{align*}
which proves \eqref{e:frag-eps}.
\end{proof}

\subsection{Necessary condition for the vanishing of $\drift_\ell$}
\label{ss:necessary}

\begin{lem} 
\label{l:relation}
If $f\in \Diff^2_+([0,1])$, then for all $x\in(0,1)$, 
$$-\nu_0(f)(x)+\tfrac12P(f)(x) = \int_0^x \ell(f)(x,y)dy$$
and 
$$-\nu_1(f)(x)+\tfrac12P(f)(x) = - \int_x^1 \ell(f)(x,y)dy$$
(where the integral are a priori only conditionally converging).
\end{lem}

\begin{proof}
We will only prove the first equality, the second being analogous. Let $x\in(0,1)$ and $\eps\in(0,x)$. Then $y\mapsto \ell(f)(x,y)$ is continuous on $[0,x-\eps]$, and 
\begin{align*}
\int_0^{x-\eps}\ell(f)(x,y)dy
&=\int_0^{x-\eps}\left(\frac{Df(x)Df(y)}{(f(x)-f(y))^2}-\frac1{(x-y)^2}\right)dy\\
&=\left[ \frac{Df(x)}{f(x)-f(y)}-\frac1{x-y}\right]_0^{x-\eps}\\
&=\frac{Df(x)}{f(x)-f(x-\eps)}-\frac1{\eps}\underbrace{-\frac{Df(x)}{f(x)}+\frac1x}_{-\nu_0(f)(x)}
\end{align*}
So we are left with proving that $\frac{Df(x)}{f(x)-f(x-\eps)}-\frac1{\eps}$ goes to $\frac12P(f)(x)$ when $\eps$ goes to $0$. And indeed, one has
$$
f(x-\eps)=f(x)-\eps Df(x)+\frac{\eps^2}2D^2f(x)+o(\eps^2),
$$
so $f(x)-f(x-\eps)=\eps Df(x)-\frac{\eps^2}2D^2f(x)+o(\eps^2)$ and thus
\begin{align*}
\frac{Df(x)}{f(x)-f(x-\eps)}-\frac1{\eps}&=\frac1\eps\left(\frac{\eps Df(x)}{f(x)-f(x-\eps)}-1\right)\\&=\frac1\eps\left(\frac1{1-\frac\eps2 P(f)(x)+o(\eps)}-1\right)=\frac12P(f)(x)+o(1), 
\end{align*}
as required.
\end{proof}

\begin{cor}
\label{c:necessary}
For all $f\in\Diff^{>2}_+[0,1]$ without interior fixed point and all $x\in(0,1)$,
$$
\mu(f)(x) = \int_0^x\ell(f)(x,y)dy-\int_x^1\ell(f)(x,y)dy\quad \text{and}\quad (\nu_1-\nu_0)(f)(x)=\int_0^1\ell (f)(x,y)dy.
$$
In particular,
$$
\|\mu(f)\|_{L^1}\le \|\ell(f)\|_{L^1}\quad \text{and}\quad \|(\nu_1-\nu_0)(f)\|_{L^1}\le \|\ell(f)\|_{L^1},
$$
and after stabilization:
$$\var\log DM_f=\drift_\mu(f)\le \drift_\ell(f)$$
and
$$|\log(Df(0)Df(1))|\le\drift_{\nu_0-\nu_1}(f)\le \drift_\ell(f).$$
\end{cor}

The inequality $|\log(Df(0)Df(1))|\le \drift_\ell(f)$ actually holds in a more general context:

\begin{prop}
\label{p:multip}
For every $f\in\mathrm{Diff}^{>2}_+([0,1])$,
$$\|\ell(f)\|_{L^1}\ge |\log (Df(0)Df(1))|$$ 
and (by stabilization) 
$$\drift_\ell(f)\ge |\log (Df(0)Df(1))|.$$
\end{prop}

\begin{proof} Given $f\in\mathrm{Diff}^{>2}_+([0,1])$,
\begin{align*}
\|\ell(f)\|_{L^1}&= \iint_{[0,1]^2}\left|\ell(f)(x,y)\right|dxdy\\
& \ge \left|\iint_{[0,1]^2}\ell(f)(x,y)dxdy\right|\\
&=\left|\int_0^1 \left[\frac{Df(x)}{f(x)-f(y)}-\frac1{x-y}\right]_{y=0}^{y=1} dx\right|\\
&=\left|\int_0^1 \left(\frac{Df(x)}{f(x)-1}-\frac1{x-1}-\frac{Df(x)}{f(x)}+\frac1x\right)dx\right|\quad\text{(since $f(0)=0$ and $f(1)=1$)}\\
&=\left|\Big[\log(1-f(x))-\log(1-x)-\log f(x)+\log x\Big]_0^1\right|\quad\text{(since $x,f(x)$ \text{lie in} $[0,1]$)}\\
&=\left|\left[\log\left(\frac{f(x)-f(1)}{x-1}\right)-\log\left(\frac{f(x)}x\right)\right]_0^1\right|\\
&=\left|\log Df(1) - \log 1-\log 1+\log Df(0)\right|
\end{align*}
as announced.
\end{proof}

Combining this with \eqref{e:loc}, we already get:
\begin{cor}
\label{c:multip}
If $f\in\Diff^{>2}_+([0,1])$ satisfies $\drift_\ell(f)=0$, then $Df(0)=\frac1{Df(1)}$ and either $f$ has no interior fixed point or $f$ has only parabolic fixed points. 
\end{cor}

\begin{proof}
The equality $Df(0)=\frac1{Df(1)}$ follows directly from the above proposition. Now assume $f$ has an interior fixed point $a$. According to Lemma \ref{l:fragment}, $\drift_{\ell,[0,a]}(f)=\drift_{\ell,[a,1]}(f)=0$. Thus, applying Proposition \ref{p:multip} to $f_{|[0,a]}$, $f_{|[a,1]}$ and $f$, we get $Df(0)=\frac1{Df(a)}=Df(1)$ and $Df(0)=\frac1{Df(1)}$ so $Df(0)=Df(1)=Df(a)=1$. 
\end{proof}

We can now prove Theorem \ref{t:criterion}. Let $f\in\mathrm{Diff}^{>2}_+([0,1])$.
\medskip

\noindent\textbf{Necessary condition.} Assume that $\drift_\ell(f)=0$. Let us first assume $f$ has a hyperbolic fixed point. Then according to the above corollary, it actually has exactly two fixed points, $0$ and $1$, and $Df(0)=\frac1{Df(1)} \neq 1$. Let $g$ be the unique Möbius map having the same derivatives as $f$ at the endpoints (such a map exists thanks to the previous equality). The diffeomorphisms $f$ and $g$ are $C^{>2}$-conjugated near the endpoints by Sternberg's linearization theorem and embed respectively in a $C^1$ flow and a smooth flow without interior fixed points. According to classical works by Mather, this directly implies that they are actually $C^{>2}$-conjugated on the whole interval.

Let us now assume that $f$ has only parabolic fixed points. Let $\mathcal{C}$ be the family of closures of connected components of $[0,1]\setminus \Fix(f)$, and let $I$ be any member of this family. According to Proposition \ref{p:fragment}, $\drift_{\ell,I}(f)=0$, which implies, thanks to Corollary \ref{c:necessary}, that $f$ embeds in a $C^1$ flow that is parabolic at the endpoints; equivalently, $V_\infty(f_{|I})=0$. Now it was proved in \cite{EN21} that $V_\infty(f)=\sum_{I\in\cC}V_\infty(f_{|I})$. Therefore $V_\infty(f)=0$, and again due to \cite[Corollary 2]{EN21}, the map $f$ is ``$C^1$ parabolically flowable'' on $[0,1]$.
\medskip

\noindent\textbf{Sufficient condition.} If $f$ is $C^{>2}$-conjugated to a Möbius transformation $g$, since $\drift_\ell$ is invariant under such conjugacies and $\ell(g^n)\equiv0$ for every $n\in\N$ according to Lemma \ref{l:mobius}, we have $\drift_\ell(f)=\drift_\ell(g)=0$. Let us now assume that $f$ embeds in a $C^1$ flow without hyperbolic fixed points. Then according to Proposition \ref{p:fragment}, if $\cC$ denotes the set of closures of connected components of $[0,1]\setminus\Fix(f)$, 
$$
\drift_{\ell,[0,1]}(f)=\sum_{I\in\cC}\drift_{\ell,I}(f).
$$
But according to Theorem \ref{t:vanish-drift-gen} applied to $c=\ell$ (replacing $[0,1]$ by any compact interval), $\drift_{\ell,I}(f)=0$ for every $I\in\cC$, which concludes the proof.

\subsection{A bi-product: non continuity of $\drift_\ell$}
\label{ss:non-cont}

The general considerations of Section \ref{s:drift} imply that $\drift_\ell$ is continuous in the $C^3$ topology at every $f\in\Diff^{3,\Delta}_+([0,1])$ (here, we consider 
$r=3$ just to fix ideas). However, this does not extend to the whole group of $C^3$ diffeomorphisms.
\begin{prop}
The map $\drift_\ell$ is not continuous from $(\mathrm{Diff}^3_+([0,1]),d_{C^3})$ to $\R_+$, nor even from $(\mathrm{Diff}^\infty_+([0,1]),d_{C^\infty})$ to $\R_+$.
\end{prop}

\begin{proof}
Start with a Möbius transformation of $\R\cup\{\infty\}$ having $0$ and $1$ as (its only) hyperbolic fixed points (with opposite derivatives $\lambda^{\pm1}$), and denote by $f$ its restriction to $[0,1]$. This diffeomorphism embeds in the flow of a smooth vector field $X$. Consider a continuous family of smooth vector fields $(X_t)_{t\in[0,1]}$, such that 
\begin{itemize}
\item $X_0=X$, 
\item $X_t$ coincides with $X_0$ near the endpoints for every $t\in[0,1]$,
\item $X_t$ has no zero in $(0,1)$ for every $t\in[0,1)$,
\item $X_1$ has a parabolic zero in $(0,1)$.
\end{itemize}
Now let $f_t$ be the time-$1$ map of $X_t$ for every $t$. All the $f_t$'s, $t\in[0,1)$, are smoothly conjugate near the endpoints and have trivial Mather invariant, so they are all smoothly conjugate to $f=f_0$ and thus have vanishing $\drift_\ell$. However, $f_1$ has an interior fixed point and hyperbolic fixed points, which prevents its $\drift_\ell$ to vanish according to Corollary \ref{c:multip}.
\end{proof}

\subsection{Expression for the Liouville drift at parabolic diffeomorphisms}

Here, we express the Liouville drift at a parabolic (non $C^1$ flowable) diffeomorphism without interior fixed point in terms of its generating vector fields (compare Theorem \ref{t:formula}):

\begin{prop}
\label{t:form-liouv}
Let $p\in(0,1)$ and $f\in \Diff^{>2,\Delta}_+([0,1])$, parabolic at the endpoints. Let $X$ and $Y$ be its left and right generating vector fields respectively, $\tau_X:x\mapsto \int_p^x\frac1X$ and $\tau_Y: x\mapsto \int_p^x\frac1Y$ be the diffeomorphism they induce from $(0,1)$ to $\R$, and $m_f=\tau_X^{-1}\tau_Y$, which is a $C^r$ diffeomorphism of $(0,1)$. Finally, let $\cD_n=\sigma([f^{-n}(p),f^{-n+1}(p)]\times[f^{-2n}(p),f^{-n+1}(p)])$. Then
$$
\drift_\ell(f)=\lim_{n\to\infty} \int_{\cD_n}|\ell(m_f)| 
$$
\end{prop}

\begin{proof}
We start just like in the proof of Theorem \ref{t:formula}. Letting $g=fh$ be the diffeomorphism given by Lemma \ref{l:cancel}, $\cD=[0,f(p)]^2\setminus[0,p]^2$ and, for every $n\in\N$, $H_n=g^{-2n}f^{2n}$, the very same reasoning yields 
$$
\drift_\ell(f)=\lim_{n\to+\infty}\int_{f^{-n}(\cD)}|\ell(H_n)|.
$$
Now observe 
$$
f^{-n}(\cD) = \cD_n \cup \sigma([f^{-n}(p),f^{-n+1}(p)]\times[0,f^{-2n}(p)]).
$$
Furthermore, if $I_k$ denotes $f^{-k}([p,f(p)]$, $\cD_n$ is a union of rectangles $I_n\times I_k$ or $I_k\times I_n$ with $0\le k\le n$. Now for every such $k$, on $I_k$,
\begin{align*}
H_n = g^{-2n}f^{2n}= g^{n-k}g^{-n}(g^{-2n+k}f^{2n-k})f^nf^{k-n}=g^{n-k}g^{-n}f^nf^{k-n}
\end{align*}
because $f^{2n-k}=g^{2n-k}$ on $f^k(f^{-k}(I_0)=[p,f(p)]$ (and even on $[p,1]$). Furthermore, just like in the proof of Theorem \ref{t:formula}, $g^{-n}f^n=m_f$ on $f^{k-n}(I_k)=I_n$, $g^{n-k}=f^{n-k}$ on $g^{-n}f^nf^{k-n}(I_k)=I_n$, and $m_f$ commutes with $f$. Therefore, in the end, $H_n = m_f$ on $I_k$ and 
$$
\int_{\cD_n}|\ell(H_n)|=\int_{\cD_n}|\ell(m_f)| .
$$
We are thus left with checking that the integral of $|\ell(H_n)|$ on $\sigma([f^{-n}(p),f^{-n+1}(p)]\times[0,f^{-2n}(p)])$ goes to $0$. Now, letting $u:(x,y)\mapsto \frac1{(x-y)^2}$, this integral is equal to
$$
2\sum_{k>2n}\int_{I_k\times I_n}\left|H_n^*u-u\right|\le 2\sum_{k>2n}\left(\int_{I_k\times I_n}|H_n^*u|+\int_{I_k\times I_n}|u|\right)
=4\sum_{k>2n}\int_{I_k\times I_n}|u|=4\int_{(0,f^{-2n}(p)]\times I_n}|u|
$$
since $H_n$ fixes the intervals $I_k$. The last integral can be computed explicitly:
$$
\int_{I_k\times I_n}|u| = \log\left(\frac{f^{-n}(p)}{f^{-n+1}(p)}\right)+\log\left(\frac{f^{-n+1}(p)-f^{-2n}(p)}{f^{-n}(p)-f^{-2n}(p)}\right).
$$
The first term on the right goes to $-\log Df(0)=0$ since $f$ was assumed parabolic, so we are left with proving that the last term goes to $0$ as well, or equivalently that $\frac{f^{-n+1}(p)-f^{-n}(p)}{f^{-n}(p)-f^{-2n}(p)}$ goes to~$0$. 
This is a direct consequence of the following lemma (applied to $x=f^{-n}(p)$), observing that the inequality
$$
0\le \frac{f^{-n+1}(p)-f^{-n}(p)}{f^{-n}(p)-f^{-2n}(p)} \le \frac{f^{-n+1}(p)-f^{-n}(p)}{f^{-n}(p)-f^{-n-N}(p)}
$$
holds for every $n\ge N\ge1$.
\end{proof}

\begin{lem}
Given $N\ge1$, there exists $\eps\in (0,1)$ such that for every $x\in(0,\eps)$,
$$0\le \frac{f(x)-x}{x-f^{-N}(x)}\le \frac2N.$$
\end{lem}

\begin{proof}
For every $x\in(0,1)$,
\begin{align*}
\frac{x-f^{-N}(x)}{f(x)-x} = \sum_{k=1}^{N}\frac{f^{-k}(f(x))-f^{-k}(x)}{f(x)-x}=\sum_{k=1}^{N}Df^{-k}(y_{k,x})
\end{align*}
for some $y_{k,x}\in[x,f(x)]$. Since, for every $k\in[\![1,N]\!]$, $Df^{-k}(0)=1$, for all $x$ sufficiently close to $0$, $Df^{-k}(y_{k,x})\ge \frac12$, which yields the desired result.
\end{proof}

\subsection{The Liouville drift vanishes at minimal circle diffeomorphisms}
\label{s:circle}

In the case of the circle, the proper definition for $\ell$ is:
$$\ell(f)(x,y) = \frac{Df(x) Df(y)}{d(f(x),f(y))^2} - \frac{1}{d(x,y)^2}.$$

The case of circle diffeomorphisms with rational rotation number should reduce to the case of the interval and will not be studied here. For the irrational case, according to Avila-Krikorian (unpublished), any $C^\infty$ circle diffeomorphism with irrational rotation number is $C^\infty$-almost-reducible, and thus has vanishing $\drift_\ell$ according to \eqref{e:upper} ($\drift_\ell$ is defined similarly on the circle). However, this can be proven directly, and more generally for $C^{>2}$ diffeomorphisms. The result below is to be compared to \cite{Na23}, where an analogous statement is established for the asymptotic variation in regularity $C^{1+ac}$. The idea of using the Hahn-Banach theorem comes from the proof therein.

\begin{prop}
Every $C^{>2}$ circle diffeomorphism with irrational rotation number has vanishing $\drift_\ell$. 
\end{prop}

\begin{proof} 
Let $f$ be a $C^{>2}$ diffeomorphism of the circle with irrational rotation number. According to Section \ref{s:drift}, abbreviating $L^1(\T^1\times \T^1)$ by $L^1$,
\begin{equation}
\drift_\ell(f)=0 \quad\Leftrightarrow\quad \ell(f)\in \overline{\{u-f^*u,u\in L^1\}}^{L^1}.
\end{equation}
Let $F$ denote the closed vector subspace on the right hand side, and assume by contradiction that $\ell(f)\notin F$. Then the Hahn-Banach theorem asserts that there exists a continuous linear form $\lambda$ on $L^1$
 such that $\lambda_{|F}\equiv 0$ and $\lambda(\ell(f))\neq0$. But $\lambda$ is of the form $v\mapsto \int_{\T^2}v K$ where $K$ is an $L^\infty$ map. The fact that $\lambda$ cancels on $F$ directly implies, by change of variable, that $K=K\circ f^{\otimes2}$ in~$L^\infty$. Let us now turn to $\lambda(\ell(f))$, which is thus equal to $\int_{\T^2} \left(\frac{Df(x)Df(y)K(f(x),f(y))}{(f(x)-f(y))^2}- \frac{K(x,y)}{(x-y)^2}\right)dxdy$. 
 
 One cannot split the integral by linearity and apply a change of variable to obtain $0$ because the two parts are not integrable, though their difference is. However, they are integrable away from the diagonal. The idea is then to split the domain of integration into two regions of the torus that are invariant under $f^{\otimes2}$. One of these regions will stay away from the diagonal, so that both terms of the difference remain integrable (and the integrals cancel each other).The other region will be its complement, which corresponds to an arbitrarily thin topological annulus along the diagonal; therein,  the integral will tend to zero as the area of the annulus goes to zero, since the map we are integrating is $L^1$ on the whole torus. 

The annulus and its complement are defined as follows. By Denjoy's theorem, $f$ is conjugate, by some homeomorphism $\varphi$, to an irrational rotation $r$, that is, $f=\varphi^{-1} r \varphi$. So naturally, if $F=f^{\otimes2}$, $R=r^{\otimes2}$ and $\Phi=\varphi^{\otimes2}$, one has $F=\Phi^{-1}R\Phi$. Now $R$ preserves the ``circles of slope 1'' defined as $\cS^1_u=\{(x,y)\in \T^2, x-y=u \,\, \mathrm{ (mod \,\, 1) }\}$, $u\in\T^1$, and 
thus any annulus $A_\eps=\cup_{-\eps\le u\le \eps}\cS^1_u$, $\eps>0$. Therefore, $F$ preserves $\Phi^{-1}(A_\eps)$ 
for every $\eps>0$, and the area of these regions goes to $0$ as $\eps$ goes to $0$ by uniform continuity of $\varphi^{-1}$. This concludes the proof.  
\end{proof}


\section{Almost-reducibility and distortion in finite regularity}
\label{s:a-r}

In \cite{EBM25}, it was proved that if $f\in\Diff^\infty_+([0,1])$ has vanishing asymptotic variation, then it is $C^\infty$-\emph{almost-reducible}, in the sense that it can be smoothly conjugated arbitrarily $C^\infty$-close to the identity. In this section, we extend this result to finite regularity. First recall the following definition.

\begin{defn}
\label{d:a-r}
Let $r\in(0,+\infty)$. An element $f\in\Diff^r_+([0,1])$ is said to be $C^r$-\emph{almost-reducible} if it can be $C^r$-conjugated arbitrarily $C^r$-close to the identity. 
\end{defn}

In \cite{EN21}, it was proved that if $f\in\Diff^{r}_+([0,1])$ is the time-$1$ map of a $C^r$ vector field non vanishing in $(0,1)$ and without hyperbolic fixed point, then it is $C^r$-almost-reducible. The main problem is that if one only assumes that $f\in\Diff^{r}_+([0,1])$ has vanishing asymptotic variation and no interior fixed point, or equivalently is the time-$1$ map of a $C^1$ parabolic vector field without interior zero, this vector field is in general only $C^{r-1}$ on $(0,1)$ and $C^1$ on $[0,1]$, and the rather elementary argument of \cite{EN21} does not work in this case. However, this issue was overcome in \cite{EBM25} in the smooth setting, thanks to the fact that, though it is globally irregular, there are regions arbitrarily close to the endpoints where the generating vector field is rather tame (cf.~Definition \ref{d:tame} below). Moreover, \cite{EBM25} deals with the nontrivial issue of possible interior fixed points.

\begin{defn}
\label{d:tame}
Let $r\in\N$ and let $X$ be a $C^1$ vector field on $[0,1]$ that is $C^r$ on the complement of its zero set and whose time-$1$ map $f$ is $C^r$ on the whole interval. We say that $X$ is \emph{$r$-tame} at $p\in X^{-1}(\{0\})$ if $f-\id$ is $C^r$-flat at $p$ and there exists $\delta\in(0,\frac1r)$ such that, arbitrarily close to $p$ (on both sides if applicable), one can find $x_0$ such that 
$$
\forall k\in[\![1,r]\!],\quad |D^kX(x_0)|\le |f(x_0)-x_0|^{1-k\delta}.
$$
\end{defn}

The analogue in finite regularity of what is proved in \cite{EBM25} can then be stated as follows:

\begin{thm}
\label{t:a-r}
Let $r\in\N$ and let $f\in\Diff^{r}_+([0,1])$ be the time-$1$ map of a $C^1$ parabolic vector field $X$. If, for every zero $p$ of $X$, either $p$ is an isolated zero at which $X$ is $C^r$, or $X$ is $r$-tame at $p$, then $f$ is $C^{r}$-almost-reducible.
\end{thm}

It directly follows from Proposition \ref{p:bounds} in the Appendix \ref{s:bounds} that if $f$ is $C^R$ with $R>2r$ and is the time-$1$ map of some $C^1$ vector field $X$, then $X$ is $r$-tame at every $C^{R}$-flat fixed point of $f$. What's more, according to \cite{Ta,Yoc}, $X$ is $C^{R-1}$ and thus $C^{r}$ on the complement (and the non-flat fixed points are naturally isolated). So we directly get the following corollary:

\begin{cor}
\label{c:a-r}
Let $2\le R\in\N$ and $f\in\Diff^{R}_+([0,1])$ with vanishing asymptotic variation. Then $f$ is $C^{r}$-almost-reducible for every $r<\frac{R}2$, and actually $C^{R-1}$-reducible if $f-\id$ is nowhere $C^R$-flat. 
\end{cor}

The central ingredients of Theorem \ref{t:a-r} are the next two propositions.

\begin{prop} 
\label{l:woITI}
Let $\eta>0$, $r\in\N$ and $X$ be a $C^r$ vector field on $[0,1]$ vanishing only at $0$ and~$1$ and with vanishing first derivative at these points. Then there exists $\lambda_X>1$ such that, for any $\lambda\ge\lambda_X$, there exists $\f\in\Diff^r_+([0,1])$ such that:
\begin{enumerate}
\item $\|\varphi_*X\|_r<\eta$,
\item on some neighborhood of each endpoint,
$\f$ coincides with the homothety of ratio $\lambda$ centered at that point.
\end{enumerate}
\end{prop}

\begin{prop}
\label{p:a-r}
Let $r \geq 2$ be an integer and let $f\in\Diff^r_+([0,1])$ be the time-1 map of a vector field $X$ that is $C^1$ on $[0,1]$, $C^r$ on $(0,1)$ and $r$-tame at $0$. Then, for every $\eta>0$, there exists $x_0<\eta$ and a $C^1$ vector field $Y$ such that: 
\begin{enumerate}
\item $Y=X$ on $[0,x_0]$,
\item $Y$ is constant equal to $X(x_0)$ on $[x_0+\eta,1)$,
\item $Y$ is $C^r$ on $(0,1)$,
\item $\|Y_{|[f^{\mp1}(x_0),x_0+\eta]}\|_{r}\le\eta$,
\item $Y$ does not vanish on $[x_0,x_0+\eta]$.
\end{enumerate}
\end{prop}

In the next section, we explain how to adapt arguments of \cite{EBM25} to deduce Theorem \ref{t:a-r} from the above propositions, and sketch the proof of Proposition \ref{l:woITI}. In \S\ref{ss:a-r}, we give a complete proof of Proposition \ref{p:a-r}, which is the only place where some special care is needed to go from infinite to finite regularity.

\subsection{From Propositions \ref{l:woITI} and \ref{p:a-r} to Theorem \ref{t:a-r}}
\label{ss:adapt}

In \cite[\S2.1]{EBM25}, one explains in detail the obstacles on the way towards the smooth analogue of Theorem \ref{t:a-r}, and how to overcome them one at a time. The difficulties and solutions are similar here, but, unfortunately, we cannot repeat such a big part of another article. We settle for a shorter overview and refer the reader to the original reference for more details (concretely, Propositions \ref{l:woITI}, \ref{p:a-r} and \ref{p:a-r-reduite} below correspond to Propositions 2.1, 2.2 and 2.3 therein, respectively). One natural and easy step, requiring no knowledge on conjugacy classes of interval diffeomorphisms, consists in reducing to the case where the only possible $r$-tame fixed points lie in the boundary. 

\begin{prop}
\label{p:a-r-reduite}
Let $r$, $f$ and $X$ be as in Theorem \ref{t:a-r} and, in addition, assume that the only possible $r$-tame fixed points lie in the boundary $\{0,1\}$. Then $f$ is $C^r$-almost-reducible. 

More precisely, for every $\eps>0$, there exists $\f\in\Diff_+^r([0,1])$ such that $d_r(\f f \f^{-1},\id)<\eps$ and such that $\f-\id$ is $C^r$-flat at every $r$-tame fixed point. 
\end{prop}

We will see shortly that, thanks to the precision at the end, one can apply this result on neighboring subintervals of $[0,1]$ in the general context of Theorem \ref{t:a-r} and obtain conjugacies that are $C^r$ even near the gluing points. Now regarding the restricted statement \ref{p:a-r-reduite}, the case where $X$ is $C^r$ and has finitely many zeroes is a straightforward consequence of Proposition~\ref{l:woITI}: one applies the latter to the closure of every connected component of the complement of the zero set, and everything glues up nicely again thanks to the specification on the boundary germs of the conjugacy in Proposition \ref{l:woITI}. (It is worth noting at this point that these boundary requirements, necessary for our purpose, substantially complicate the proofs...). The general case uses a combination of Propositions \ref{l:woITI} and \ref{p:a-r}, and its proof, as well as that of \ref{l:woITI}, will be sketched below. Before that, naturally, we will need to explain how to guarantee that two relatively simple interval diffeomorphisms (or their generating vector fields) are conjugated (due to this navigation between diffeomorphisms and vector fields, we will sometimes abusively use the word \emph{conjugacy} for vector fields, instead of the standard terminology of push-forward or pull-back). 

\paragraph{Reduction to Proposition \ref{p:a-r-reduite} (sketch).}

Let $f\in\Diff^{r}_+([0,1])$ satisfy the hypotheses of Theorem \ref{t:a-r}. Given $\eps>0$, one easily shows the existence of a finite subdivision $a_0=0<\dots<a_n=1$ of $[0,1]$ such that the generating vector field $X$ is tame at every interior $a_i$ and, for each $i$, either $X$ is $C^r$ and has no accumulation of zeroes in the interior of $I_i=[a_i,a_{i+1}]$ or $f$ is $C^r$-close to $\id$ on $I_i$ (for this, simply note that, near an accumulation point of $r$-tame zeroes of $X$, $f$ must be $C^r$-close to the identity). One can then find a $C^r$ conjugate of $f$ that is $C^r$-close to $\id$ by applying Proposition \ref{p:a-r-reduite} on the first subintervals and letting $f$ unchanged on the others. We refer to \cite[\S2.3]{EBM25} for the details.

\paragraph{Conjugacy criterion.} Classical works of \cite{Mather} show that if two $C^s$ vector fields on $[0,1]$ without interior zero have the same germs at the endpoints, then they are globally $C^s$-conjugated by a diffeomorphism that is the identity near one of the endpoints and a flow map of both vector fields near the other. Adding some extra ``synchronisation condition'', one can guarantee that the latter flow map is actually the identity (cf. \cite[\S1]{EBM25} for more details).

\paragraph{Sketch of proof of Proposition \ref{l:woITI}.} First, conjugating the given vector field $X$ by a smooth diffeomorphism which is a homothety of sufficiently big ratio near the boundary, we get a new vector field $\bar X$ with small derivatives up to order $r$ near the endpoints. One can interpolate between the restrictions of $\bar X$ near $0$ and $1$ to get a $C^r$-small vector field $\tilde X$ on the whole interval. According to the previous paragraph, $\tilde X$ is $C^r$ conjugated to $\bar X$, but by a diffeomorphism that coincides with the identity near one of the endpoints and with a flow map near the other, while we would like it to coincide with the identity near both. Some extra care in both steps (initial conjugation and interpolation) allows to fix this issue (cf. \cite[\S2.4]{EBM25} for the details).

\paragraph{Sketch of proof of Proposition \ref{p:a-r-reduite} assuming Propositions \ref{l:woITI} and \ref{p:a-r}.} To fix ideas, let us consider the case where $X$ is $r$-tame at $0$ and $1$. Applying Proposition \ref{p:a-r} near both endpoints, we get two vector fields $X_0$ and $X_1$ that coincide with $X$ near $0$ and $1$, respectively (say on $[0,x_0]$ and $[x_1,1]$,  respectively), are $C^r$-small outside the corresponding neighborhood, and are constant outside a bigger neighborhood. If the initial vector field had no zero in the interior, one can interpolate between the constant parts of $X_0$ and $X_1$ to get a $C^1$ vector field $\tilde X$ which is $C^r$ small except maybe in the boundary region where it coincides with the initial $X$, where $f$ is $C^r$-close to the identity. The time-1 map $\tilde f$ of $\tilde X$ is thus $C^r$-close to the identity, and it is $C^1$-conjugated to $f$, and even $C^r$-conjugated if one puts some extra care in the interpolation to fulfill the ``synchronization condition'' alluded to above. 

If the initial vector field did have zeroes in the interior (necessarily in finite number in $[x_0,x_1]$ due 
to the ``isolation'' condition), one should first conjugate $X_{|[x_0,x_1]}$ using Proposition \ref{l:woITI} and a 
contracting homothety, and then interpolate between this new vector field and $X_0$ and $X_1$. 
For more details, see \cite[\S2.6]{EBM25}.  

\subsection{Proof of Proposition \ref{p:a-r}}
\label{ss:a-r}

The following is a refinement of the proof of Proposition~2.7 in \cite{EBM25} with a particular care given to the optimality of the exponents.

\begin{proof}
Let $\chi:[0,+\infty) \rightarrow \mathbb{R}$ be a $C^{\infty}$ function that is equal to $1$ on $[0,\frac{1}{2}]$ and to $0$ on $[1,+\infty)$, with $0 \leq \chi \leq 1$. Let $M_r=\left\| \chi \right\|_r$. By assumption, there exists $\delta\in(0,\frac1r)$ such that, arbitrarily close to $0$, one can find $x_0>0$ such that 
\begin{equation}
\label{e:x0}
\forall k\in[\![1,r]\!],\quad |D^kX(x_0)|\le |f(x_0)-x_0|^{1-k\delta}.
\end{equation}
Since $1-r\delta>0$, one can pick $\eps>0$ small enough that $1-r\delta-(r-1)\eps$ is still positive. 
Take $x_0$ such that \eqref{e:x0} holds and small enough so that 
$$
X(x_0)+ r2^r M_r (f(x_0)-x_0)^{1-r\delta-(r-1)\eps} < \eta
$$
and
$$ 
\frac{X(x_0)}{f(x_0)-x_0}>e^{(f(x_0)-x_0)^{\eps}}-1.
$$ 
For the second inequality, observe that the right-hand side expression
is equivalent to $(f(x_0)-x_0)^{\eps}$ and thus goes to $0$ when $x_0$ goes to $0$, while the left-hand side goes to $1$ (cf. Lemma \ref{t:szek} in the appendix). 
Take 
\begin{equation}
\label{e:b}
b=(f(x_0)-x_0)^{\delta+\eps}.
\end{equation}
For any $i \in [\![0,r]\!]$, let $a_i=D^i X(x_0)$, so that $|a_0|=|X(x_0)|$. 
By the choice of $x_0$,
\begin{equation}
\label{e:ai}
|a_i|\le (f(x_0)-x_0)^{1-i\delta},
\end{equation}
and thus 
\begin{equation}
\label{e:aibi}
|a_i|b^i\le (f(x_0)-x_0)^{1+i\eps}.
\end{equation}
We define $Y$ by setting $Y(x)=X(x)$ if $x \leq x_0$ and
$$
Y(x)=a_0+ \chi \left(\frac{x-x_0}{b} \right)\sum_{i=1}^{r}\frac{a_i}{i!}(x-x_0)^i
$$
if $x \geq x_0$. It is clear that the vector field $Y$ satisfies the three first properties. Let us now deal with property 4.

For any $i \in [\![1,r]\!]$, let $f_i$ be defined for $y \in \mathbb{R}$ by
$$
f_i(y)=\frac{a_i}{i!} y^i \chi \left(\frac{y}{b}\right).
$$
Then for any $k \in [\![0,r]\!]$,
$$
D^kf_i(y)=\sum_{j=0}^{i} \binom{k}{j}  \frac{a_i}{(i-j)!}\frac{y^{i-j}}{b^{k-j}} D^{k-j}\chi \left(\frac{y}{b}\right).
$$
Hence, as this function is supported in $[0,b]$,
$$
\begin{array}{rcl}
\left\| D^k f_i \right\|_{\infty} & \leq & \displaystyle \sum_{j=0}^{i} \binom{k}{j}  |a_i|b^{i-k} M_r \\
 & \leq & 2^{k} M_r (f(x_0)-x_0)^{1+i\eps-k(\delta+\eps)} \quad\text{according to \eqref{e:aibi} and \eqref{e:b}}\\
 & \leq & 2^{r} M_r (f(x_0)-x_0)^{1-r\delta-(r-1)\eps}.
\end{array}
$$
From this and from our choice of $x_0$, we deduce the fourth property.
\medskip

It remains to check the last point. If $Y$ vanishes at some point $x_1$, then $x_1=x_0+y_1$ belongs to $[x_0,x_0+b]$. As $Y(x_1)=0$, we have 
$$ a_0=-\sum_{i=1}^{k} f_i(y_1),$$
so that
$$
X(x_0) 
  = |a_0| 
 \leq  \displaystyle \sum_{i=1}^{k} |f_i(y_1)| 
 \leq  \displaystyle \sum_{i=1}^{k} \frac{|a_i|}{i!}b^i 
 \leq  (f(x_0)-x_0)\displaystyle \sum_{i=1}^{k} \frac{1}{i!} (f(x_0)-x_0)^{i\eps}, 
$$
according to \eqref{e:aibi}. As a consequence, 
$$\frac{X(x_0)}{f(x_0)-x_0}\le e^{(f(x_0)-x_0)^{\eps}}-1,$$
which is not possible by our choice of $x_0$.
\end{proof}

\subsection{From almost-reducibility to distortion}
\label{ss:a-r-dist}

Here, we give an elementary proof of Theorem \ref{t:dist-fini} (recalled below), based on finite regularity analogues of statements of \cite{EBM25} whose proofs are sketched in Appendix \ref{s:perfect}.

\begin{thmF}
If $f\in\Diff^r_c(\R)$ is $C^r$-almost-reducible for a certain $r \geq 3$, then it is $C^{r-2}$-distorted. 
\end{thmF}

Here are the fundamental ingredients: 

\begin{prop} 
\label{p:GSM2}
Let $r\ge3$. There exists a sequence of positive numbers $(\eta_n)_{n \geq 0}$ such that, for any sequence of diffeomorphisms $(f_n)_{n \geq 0}$ in $\Diff^r_c(\R)$ with a common compact support and satisfying $d_{C^r}(f_n,\id) < \eta_n$, the following property holds: there exists a finite subset $S$ of  $\Diff^{r-2}_c(\R)$ such that, for any $n$, the 
diffeomorphism $f_n$ belongs to $ \langle S \rangle$ and
$$\ell_S(f_n) \leq 70n+70.$$
\end{prop}

\begin{lem} 
\label{Lem:smallpieces}
Let $r\neq2$. For every diffeomorphism $\varphi$ in $G=\Diff^r_c(\R)$ and every $\delta>0$, there exists a subset $S\subset G$ with $|S| \leq 14$ such that:
\begin{enumerate}
\item The subset $S$ consists of diffeomorphisms that are $\delta$-close to the identity.
\item The diffeomorphism $\varphi$ belongs to $\langle S \rangle$.
\end{enumerate}
\end{lem}

Let us now prove Theorem \ref{t:dist-fini}. Fix $r\ge3$. Let $f$ be a $C^r$-almost-reducible $C^r$-diffeomorphism, meaning that $f$ (and thus its iterates) can be conjugated arbitrarily close to the identity by diffeomorphisms with a common compact support in the interior of a compact interval $I$. The idea is to apply Proposition \ref{p:GSM2} to a sequence $(f_n)$ consisting both of conjugates of big iterates of $f$ that are close to the identity and ``small pieces'' (close to the identity) provided by Lemma~\ref{Lem:smallpieces} for each of the previous conjugating diffeomorphisms. 
Concretely, let:
\begin{itemize}
\item $(n_k)$ be an increasing sequence of integers such that $70k+70=o(n_k)$ (which will be the exponents of the iterates alluded to above),
\item $(\eta_n)$ be the sequence given by Proposition \ref{p:GSM2},
\end{itemize}
and, for every $k\in\N$, let 
\begin{itemize}
\item $\delta_k=\min\{\eta_n, n\le 15k\}$,
\item $\varphi_k\in \Diff^r_c(I)$ be such that $d_{C^r}(\varphi_kf^{n_k}\varphi_k^{-1},\id)<\delta_k$,
\item $S_k\subset \Diff^r_c(I)$ be a family of $14$ elements obtained by applying Lemma \ref{Lem:smallpieces} to $\varphi_k$ and $\delta_k$,
\item $S'_k:=S_k\cup\{\varphi_kf^{n_k}\varphi_k^{-1}\}$ (which lies in the $\delta_k$-neighborhood of the identity for the $C^r$ topology).
\end{itemize}
Define a sequence $(f_n)$ of $C^r$ diffeomorphisms supported in $I$ by enumerating the elements of $S'_1$, then $S'_2$, etc. More precisely, $\{f_1,\dots,f_{14}\}=S_1$, $f_{15}=\varphi_1f^{n_1}\varphi_1^{-1}$, $\{f_{16},\dots,f_{31}\}=S_2$, $f_{32}=\varphi_2f^{n_2}\varphi_2^{-1}$, etc. By construction, for every $n\in\N$, we have 
$d_{C^r}(f_n,\id)<\eta_n$, so Proposition~\ref{p:GSM2} applies and provides a finite family $S\subset \Diff^{r-2}_c(\R)$ such that, for every $n\in\N$, $f_n\in\,\langle S\rangle$ and ${|f_n|}_S\le 70 n + 70$. Since $\varphi_k\in\,\langle S_k\rangle$ for every $k\in\N$, a fortiori we have $\varphi_k\in\,\langle S\rangle$. Let $p_k=|\varphi_k|_S$. Then, for all $k\in\N$, $f^{n_k} = \varphi_k\circ f_{15k}\circ\varphi_k^{-1}$. Therefore,
\begin{align*}
\frac{|f^{p_kn_k}|_S}{p_kn_k} 
= \frac{1}{p_kn_k}\left|\varphi_k^{-1}\,f_{15k}^{p_k}\;\varphi_k\right|_S
\le  \frac{1}{p_kn_k}\left(2p_k + p_k(70k+70)\right)
\le\frac1{n_k}(70k+72)\xrightarrow[k\to\infty]{}0,
\end{align*}
which concludes the proof.


\section{Appendix: Bounds on the generating vector fields}
\label{s:bounds}

As was already mentioned, if a $C^r$ diffeomorphism of the interval is the time-$1$ map of a $C^1$ vector field, this (unique) vector field is not necessarily $C^r$. It is always $C^{r-1}$ on the complement of the set of ``flat'' fixed points (where $f-\id$ is $C^r$-flat) according to \cite{Ta,Yoc}, but can sometimes be only $C^1$ at these fixed points. This phenomenon was highlighted by Sergeraert in \cite{Se}, where he nevertheless obtained, in the $C^\infty$ setting and in the absence of interior fixed points, some control on the vector field in terms of the diffeomorphism it generates, allowing him to prove the smoothness of the vector field under some extra ``non oscillation'' condition. In \cite[Prop. 2.17]{BE16}, the authors observed that Sergeraert's estimates imply, in the general case (with no extra condition and with possibly interior fixed points), that there are regions arbitrarily close to the fixed points where the generating vector field is rather tame. Here we prove a refined version of this result in finite regularity, Proposition \ref{p:bounds} below. The difference is that we keep track of the exact exponents in the estimates, but apart from that, the proof is exactly the same, and no new idea is involved. For readability reasons, we restrict to diffeomorphisms without interior 
fixed points, but just like in \cite{BE16}, the generalization to the case with interior fixed points is straightforward, though tedious, as it requires considering $f$ or $f^{-1}$ on the various complementary intervals of the set of fixed points, depending on which ``pushes to the left/right''.

\begin{prop}
\label{p:bounds}
Let $2\le r<R\in\N$ and $f$ be an element of $\Diff^R_+([0,1])$ which is $C^R$-tangent to the identity at $0$ and is the time-$1$ map of a $C^1$ 
vector field $X$ that is ``negative'' on $(0,1)$. Then there exists $C=C(r,R,f)$ such that, arbitrarily close to $0$, one can find $x_0>0$ such that 
$$ \left\| X_{|[f^{2}(x_0), f^{-2}(x_0)]} \right\|_r \leq C x_0^r\,|f(x_0)-x_0|^{1-\frac{2r}R}.$$
\end{prop}

 The proof is a combination of the following lemmas. Lemma \ref{l:dist} is elementary and Lemma~\ref{t:szek} is quite standard. They correspond to Lemmas 2.18  and 2.19 in \cite{BE16}, for example (in a much more general context with interior fixed points; for the case without interior fixed points, see \cite[Lemma 2.9]{Se});  we will not repeat their proof here. Lemma \ref{p:estb} is the difficult part, and its proof will take the three whole sections \ref{ss:prep}, \ref{ss:estb1} and \ref{ss:estb2}.  It is a refined version of Lemma 2.20 from \cite{BE16} in the case without interior fixed points, which corresponds to Lemma 3.6 in \cite{Se}.

\begin{lem}
\label{l:dist} 
Let $f$ be \emph{any} $C^1$ diffeomorphism of $[0,1)$ satisfying $Df(0)=1$. Then
$$\sup_{y\in[x,f^{\pm2}(x)]}\left|\frac{f(y)-y}{f(x)-x}-1\right|\underset{{x\to 0}\atop{x\notin\Fix(f)}}{\to}0.$$
\end{lem}

\begin{lem}
\label{t:szek} 
Let $f$ and $X$ be as in Proposition \ref{p:bounds}. Then
\begin{itemize} 
\item $\log\left|\frac{X}{f-\id}\right|$ is bounded on $(0,1)$,
\item $X(x)\underset{x\to 0}{\sim}f(x)-x.$ 
\end{itemize}
\end{lem} 

\begin{lem}
\label{p:estb}
Let $R$, $f$ and $X$ be as in Proposition \ref{p:bounds}. Then, for all $n\in[\![0,R-1]\!]$, there exists $C=C(n,R,f)$ such that, for all $x$ close enough to $0$,
$$|X^{n-1}(x)D^{n}X(x)|\le Cx^n\norm{f-\id}_{0,[0,x]}^{n(1-\frac2R)}.$$
\end{lem}

\subsection{Proof of Proposition \ref{p:bounds}}

This is a slightly adapted version of the proof of Proposition 2.17 in \cite{BE16}. {\color{black}Let $r$, $R$, $f$ and $X$ be as in Proposition \ref{p:bounds}. According to Lemma~\ref{p:estb}, there exist $C>1$ and $x_1>0$ such that, for all $x\in(0,x_1]$ and all $k\in[\![0,R-1]\!]$,
\begin{equation}\label{e:estb1}
|X^{k-1}(x)D^{k}X(x)|\le Cx^k\,\norm{f-\id}_{0,[0,x]}^{k(1-\frac2R)}.
\end{equation}
Pick $x_0\in(0,x_1]$ satisfying $|f(x_0)-x_0| = \norm{f-\id}_{0,[0,x_0]}$. Such an $x_0$ can be chosen arbitrarily close to $0$.For all $x\in [f^{2}(x_0),f^{-2}(x_0)],$ 
\begin{equation}
\label{e:xxo}
\norm{f-\id}_{0,[0,x]} \le\norm{f-\id}_{0,[0,f^{-2}(x_0)]}. 
\end{equation}
Therefore, for all $k\in[\![0,R-1]\!]$, 
\begin{align*}
|D^kX(x)|
&\le Cx^k\,\frac{\norm{f-\id}_{0,[0,x]}^{k(1-\frac2R)}}{|X(x)|^{k-1}}\quad \text{by \eqref{e:estb1}}\\
&\le Cx^k\,\frac{\norm{f-\id}_{0,[0,f^{ -2}(x_0)]}^{k(1-\frac2R)}}{|X(x)|^{k-1}}\quad \text{by \eqref{e:xxo}}\\
&=Cx_0^k\times\left(\frac{x}{x_0}\right)^k\times\left(\frac{\norm{f-\id}_{0,[0,f^{ -2}(x_0)]}}{\norm{f-\id}_{0,[0,x_0]}}\right)^{k(1-\frac2R)}\times \frac{|f(x_0)-x_0|^{k(1-\frac2R)}}{|f(x_0)-x_0|^{k-1}}\\
&\hspace{9cm}
\times\left|\frac{f(x_0)-x_0}{f(x)-x}\right|^{k-1}
\times \left|\frac{f(x)-x}{X(x)}\right|^{k-1},
\end{align*}
where we have used the equality of the first denominator and the second numerator given by the choice of $x_0$.
Thus, 
$$
\|D^kX\|_{0,[f^{2}(x_0),f^{-2}(x_0)]}
\le C |f(x_0)-x_0|^{1-\frac{2k}R}\times A(x_0)^k\times B(x_0)^{k(1-\frac2R)}\times C(x_0)^{k-1}\times D(x_0)^{k-1},
$$
where
$$
A(x_0) = \max_{[f^2(x_0),f^{-2}(x_0)]}\left|\frac{x}{x_0}\right|
=1+\max\left(\frac{x_0-f^2(x_0)}{x_0},\frac{f^{-2}(x_0)-x_0}{x_0}\right)
\xrightarrow[x_0\to0]{}1
$$
because $f^2-\id$ and $f^{-2}-\id$ are $C^R$-flat at $0$,
$$
B(x_0):=\max_{[x_0,f^{-2}(x_0)]}\left|\frac{f(x)-x}{f(x_0)-x_0}\right| 
\xrightarrow[x_0\to0]{}1
$$
and 
$$
C(x_0):=\max_{[x_0,f^{-2}(x_0)]}\left|\frac{f(x_0)-x_0}{f(x)-x}\right| 
\xrightarrow[x_0\to0]{}1
$$
according to Lemma \ref{l:dist}, and
$$
D(x_0):=\max_{[x_0,f^{-2}(x_0)]}\left|\frac{f(x)-x}{X(x)}\right| 
\xrightarrow[x_0\to0]{}1
$$
according to Lemma \ref{t:szek}. Since $x_0$ can be chosen arbitrarily close to $0$, this concludes the proof of Proposition \ref{p:bounds}.}

\subsection{Preparation for the proof of Lemma \ref{p:estb}: nice formulas for $X^{k-1}D^kX$}
\label{ss:prep}

We denote by $L_X$ the Lie derivative along $X$: for a $C^1$ map $\varphi$, 
$$
L_X\varphi:=X\cdot D\varphi.
$$ 
Let us introduce the following functions, defined on $(0,1)$:  
$$
\forall n\in[\![1,R-1]\!],\quad 
\mu_n=X^{n-1}D^{n}X,\quad  \Phi_n= (L_X)^{n-1}DX
\quad\text{and}\quad
\varphi_n=-(L_X)^{n}\log Df.
$$
With these notations, 
\begin{equation}
\begin{cases}
\mu_1=\Phi_1,\\
L_X f^{i}= X\circ f^{i}\label{e:lie}\quad \forall i\in\Z,\\
\mu_{n+1}=L_X\mu_{n}-(n-1)\mu_1\mu_{n}\quad \forall n\ge 1.
\end{cases}
\end{equation}
By induction, this leads to the following lemma, which is due to Sergeraert \cite{Se} and appears verbatim in \cite{BE16} (we will not provide a proof here). In this statement (for $n=1$), a polynomial (function) in $0$ variables composed with a $0$-tuple of functions (resp. of monomials in one variable) is to be understood as a \emph{constant} (resp. a constant polynomial in $1$ variable). We adopt this (controversial) convention only to make the statement simpler.

\begin{lem}
\label{l:alg} 
For all $n\ge 1$, 
$$\mu_n=\Phi_n-P_n(\mu_1,...,\mu_{n-1})\quad\text{on $(0,1)$}$$
and
$$\varphi_n=-\sum_{q=0}^{n-1} D^q\log Df\times X^{q+1}\times Q_{n,q}(\mu_1,...,\mu_{n-1})\quad\text{on $[0,1]\setminus\Fix(f)$},$$
for some polynomials $P_n$ and $Q_{n,q}$ in $n-1$ variables, \emph{independent of $f$}, with nonnegative (integer) coefficients, satisfying
\begin{equation}
\label{e:alg}
P_n(X,...,X^{n-1}) = \alpha_n X^{n}
\quad\text{and}\quad 
Q_{n,q}(X,...,X^{n-1})=\beta_{n,q}X^{n-1-q}\tag{$*_n$}
\end{equation}
 for some $\alpha_n,\beta_{n,q}\in \N$. 
\end{lem}

To prove Lemma \ref{p:estb}, we combine these inductive formulas to the following estimates proved by Sergeraert in \cite{Se} in the $C^\infty$ case using Hadamard inequalities, and whose finite regularity versions are straightforward. 

\begin{lem}[cf. \cite{Se}, 3.3]
\label{l:had}
Let $g\in C^k([0,1],\R)$ be $C^k$-flat at $0$. Then, for all $n\le k$, 
$$\norm{g}_{n,[0,x]}  \underset{x\to 0}{=}O(\norm{g}_{0,[0,x]}^{1-\frac{n}k}).$$
\end{lem}

\begin{cor}[cf. \cite{Se}, {3.4}]
\label{c:had}
Let $f$ be a $C^R$ diffeomorphism of $[0,1]$ such that $f-\id$ is $C^R$-flat at $0$. Then, for all $n\le R$, 
$$\norm{f-\id}_{n,[0,x]} 
\underset{x\to 0}{=}O(\norm{f-\id}_{0,[0,x]}^{1-\frac{n}{R}}).$$
Moreover, $\log Df$ is $C^{R-1}$-flat at $0$, hence for all $n\le R-1$, 
$$\norm{\log Df}_{n,[0,x]}  
\underset{x\to 0}{=}O(\norm{\log Df}_{0,[0,x]}^{1-\frac{n}{R-1}}).$$
\end{cor}

Note that under the hypothesis above, $\log Df$ and $Df-1=D(f-\id)$ are equivalent near~$0$, 
so applying the first estimate above with $n=1$, we get:

\begin{lem}
\label{l:no-se}
Let $f$ be a $C^R$ diffeomorphism such that $f-\id$ is $C^R$-flat at $0$. Then 
$$\norm{\log Df}_{0,[0,x]} 
\underset{x\to 0}{=}O(\norm{f-\id}_{0,[0,x]}^{1-\frac{1}{R}}).$$
\end{lem}

We also admit the following formula for $DX$ on $(0,1)$ (in particular, the convergence of the series) under the hypotheses of Proposition \ref{p:bounds} (it is enough to assume $f$ of class $C^2$ here, cf. for example \cite{Yoc}):
\begin{equation}
\label{e:Dxi}
DX =-\sum_{i=0}^{+\infty}(L_X\log Df)\circ f^{i}.
\end{equation}

The estimates above, first used to control $\mu_1=DX$ and then injected in the formulas of Lemma~\ref{l:alg}, result, by induction, in the following Lemma, which implies Lemma \ref{p:estb} by combining $(i)_{n+1}$ for $n\in[\![0,R-1]\!]$ to Lemma \ref{l:no-se}, observing that $(1-\frac1{R-1})(1-\frac1R)=1-\frac2R$. This induction is carried out in Sections \ref{ss:estb1} and \ref{ss:estb2}. We set $\mu_0=0$.

\begin{lem}
\label{l:estb}
Let $r, R, f, X$ be as in Proposition \ref{p:bounds}.
 Then for all $n\in[\![1,R-1]\!]$, 
\begin{enumerate}[label=$(\roman*)_n$]
\item $\mu_{n-1}(x)\underset{x\to 0}{=}O\left( x\norm{\log Df}_{0,[0,x]}^{1-\frac{1}{R-1}}\right)^{n-1}$,
\item $|\f_n(x)|\underset{x\to 0}{=}O\left( |X(x)|\times x^{n-1}\times\left( \norm{\log Df}_{0,[0,x]}^{1-\frac{1}{R-1}}\right)^{n}\right)$,
\item $\Phi_n= \sum_{i= 0}^{+\infty} \f_n\circ f^{i} \quad\text{ on $(0,1)$}$,
\item $|\Phi_n(x)|\underset{{x\to 0}}{=}O\left( \left( x\norm{\log Df}_{0,[0,x]}^{1-\frac{1}{R-1}}\right)^{n}\right)$, 
\end{enumerate}
and $(i)_R$ also holds.
\end{lem}

\subsection{Proof of Lemma \ref{l:estb}: base case, $n=1$}
\label{ss:estb1}

We check each property below:
\medskip

\noindent$(i)_1$ : This is straightforward since $\mu_0=0$.\bigskip

\noindent$(ii)_1$ : For all $x\in (0,1)$, 
$$|\f_1(x) |=| (D\log Df\times  X)(x) |\le | X(x)|\times\norm{D\log Df}_{0,[0,x]}\underset{x\to 0}{=}O\left(| X(x)|\times\norm{\log Df}^{1-\frac1{R-1}}_{0,[0,x]}\right),$$
according to Corollary \ref{c:had}.\bigskip

\noindent$(iii)_1$ : This is exactly Formula \eqref{e:Dxi}.\bigskip

\noindent$(iv)_1$ : According to $(ii)_1$ above and Lemma \ref{t:szek}, there exist $C>0$ and $x_1\in (0,1)$ such that, for all $x\in(0, x_1)$, 
\begin{equation}
\label{e:iio}
|\f_1(x) |\le C | X(x)|\times\norm{\log Df}^{1-\frac1{R-1}}_{0,[0,x]}\quad\text{and}\quad \left|\frac{ X(x)}{f(x)-x}\right|\le 2.
\end{equation}
For all such $x$, one has $f^{i}(x)\le x\le x_1$ for all $i\in\N$, hence
\begin{align*}
\label{e:iioi}
|\f_1(f^{i}(x)) |&\le C | X(f^{i}(x))|\times\norm{\log Df}^{1-\frac1{R-1}}_{0,[0,f^{i}(x)]}\\
&\le 2C\; |(f-\id)\circ f^{i}(x)|\times\norm{\log Df}^{1-\frac1{R-1}}_{0,[0,x]}\\
&\le 2C \;\left(f^{i}(x)-f^{(i+1)}(x)\right)\times\norm{\log Df}^{1-\frac1{R-1}}_{0,[0,x]}.
\end{align*}
As a consequence,
\begin{equation*}
|\Phi_1(x)|=\left|\sum_{i=0}^{+\infty}
    \f_1\circ f^{i}(x)\right|\le 2C  \norm{\log Df}^{1-\frac1{R-1}}_{0,[0,x]}\underbrace{\sum_{i=0}^{+\infty}(f^{i}(x)-f^{i+1}(x))}_{\le x}.
\end{equation*}
which concludes the proof of $(iv)_1$.

\subsection{Proof of Lemma \ref{l:estb}: inductive step}
\label{ss:estb2} 

Let $n\in[\![1,R-1]\!]$ (this is not a mistake, we really mean $R-1$ and not $R-2$ because we want to prove that $(i)_R$ holds too). Assume that $(i)_q$ to $(iv)_q$ are satisfied for all $q\le n$. 
\medskip

\noindent$(i)_{n+1}$ : According to $(iv)_n$ and $(i)_{2\,\text{to}\,n}$, there exist $C>0$ and $x_1\in(0,1]$ such that, for all $x\in(0,x_1)$, 
\begin{equation}
\label{e:phin}
|\Phi_n(x)|\le C\left( x\norm{\log Df}_{0,[0,x]}^{1-\frac{1}{R-1}}\right)^{n}\quad\text{and}\quad |\mu_{k}(x)|\le C^{k} \left( x\norm{\log Df}_{0,[0,x]}^{1-\frac{1}{R-1}}\right)^{k}\quad \forall k\in[\![1,n-1]\!].
\end{equation}
For all such $x$,
\begin{align*}
|\mu_{n}(x)| 
&= |\Phi_n(x) - P_n(\mu_1(x),...,\mu_{n-1}(x))|\quad\text{according to Lemma \ref{l:alg},}\\
& \le |\Phi_n(x)| + P_n(|\mu_1(x)|,...,|\mu_{n-1}(x)|)\quad\text{since the coefficients of $P_n$ are positive,}\\
&\le Cx^n\norm{\log Df}^{n-\frac{n}{R-1}}_{0,[0,x]}+ P_n\left(Cx\norm{\log Df}^{1-\frac{1}{R-1}}_{0,[0,x]},...,\left(Cx\norm{\log Df}^{1-\frac{1}{R-1}}_{0,[0,x]}\right)^{n-1}\right)\;\text{by \eqref{e:phin}}\\
&\le Cx^n\norm{\log Df}^{n-\frac{n}{R-1}}_{0,[0,x]} + \alpha_{n}\left(Cx\norm{\log Df}^{1-\frac{1}{R-1}}_{0,[0,x]}\right)^{n}\;\text{according to Lemma \ref{l:alg}}\\
&\le (C+\alpha_nC^{n})x^n \norm{\log Df}^{n-\frac{n}{R-1}}_{0,[0,x]},
\end{align*} 
which proves $(i)_{n+1}$.
\bigskip

Let now $n\in[\![1,R-2]\!]$, and assume $(i)_q$ to $(iv)_q$ are satisfied for all $q\le n$.
\medskip

\noindent$(ii)_{n+1}$ : According to $(i)_{2\,\text{to}\,n+1}$, there exist $C>0$ and $x_1\in(0,1]$ such that, for all $x\in(0,x_1]$, 
\begin{equation}
\label{e:mum}
|\mu_k(x)|\le C^k \norm{\log Df}^{k-\frac{k}{R-1}}_{0,[0,x]}\quad \forall k\in[\![1,n]\!].
\end{equation}
In particular, for all such $x$ and all $q\in[\![0,n]\!]$,
\begin{align*}
\left| Q_{n+1,q}(\mu_1,...,\mu_{n})(x)\right|
&\le Q_{n+1,q}\left( |\mu_1(x)|,...,|\mu_{n}(x)|\right)\,\text{\,\,\, since the coef. of $Q_{n+1,q}$ are $\ge0$,}\\
&\le Q_{n+1,q}\left(Cx\norm{\log Df}^{1-\frac{1}{R-1}}_{0,[0,x]},...,\left(Cx\norm{\log Df}^{1-\frac{1}{R-1}}_{0,[0,x]}\right)^{n}\right)\,\text{\,\, by \eqref{e:mum}}\\
&=\beta_{n+1,q}\left(Cx\norm{\log Df}^{1-\frac{1}{R-1}}_{0,[0,x]}\right)^{n-q}\;\text{\, by Lemma \ref{l:alg}}\\
&\underset{x\to 0}{=}O\left( \left(x\norm{\log Df}^{1-\frac{1}{R-1}}_{0,[0,x]}\right)^{n-q}\right).
\end{align*}
Now, according to Lemma \ref{l:alg}, for all $x\in(0,1)$,
\begin{align*}
\label{e:fin}
|\f_{n+1}(x)|
&=\left|\sum_{q=0}^{n} D^qPf(x)\times  X^{q+1}(x)\times Q_{n+1,q}(\mu_1,...,\mu_{n})(x)\right|\notag\\
&\le | X(x)|\times \sum_{q=0}^{n}\underbrace{\left|D^qLf(x) \right|}_{{\underset{x \to 0}{=}O\left(\norm{\log Df}^{1-\frac{q+1}{R-1}}_{0,[0,x]}\right)}\atop {\text{according to \ref{c:had}}}}\times |  X^{q}(x)| 
\times \underbrace{\left| Q_{n+1,q}(\mu_1,...,\mu_{n})(x)\right|.}_{{\underset{x \to 0}{=}O\left(\left(x\norm{\log Df}^{1-\frac{1}{R-1}}_{0,[0,x]}\right)^{n-q}\right)}\atop{\text{as we just saw}}}
\end{align*}
Now observe that $|X(x)|^q=O(|f(x)-x|^q)$ and 
$$|f(x)-x|=\left|\int_0^xD(f-\id)(y)dy\right|=O(x\|D(f-\id)\|_{0,[0,x]})=O(x\|\log Df\|_{0,[0,x]}).$$
Adding the exponents, one precisely gets that 
$$\quad|\f_{n+1}(x)|\underset{{x\to 0}}{=}O\left( | X(x)|\times x^n\times \left(\norm{\log Df}^{1-\frac{1}{R-1})}_{0,[0,x]}\right)^{n+1}\right),$$ 
which proves $(ii)_{n+1}$.
\medskip

Note that, more generally:
\begin{claim}
\label{c:fibound}
The quotient
$\frac{\varphi_{n+1}}{ X}$ is bounded on $(0,1)$. 
\end{claim}

\begin{proof}
According to Lemma \ref{l:alg}, on $(0,1)$,
\begin{equation*}
\label{e:fibound}
\frac{\varphi_{n+1}}{ X}=-\sum_{q=0}^{n} D^{q+1}\log Df\times  X^{q}\times Q_{n+1,q}(\mu_1,...,\mu_{n}).
\end{equation*}
For all $q\in[\![0,n]\!]\subset [\![0,R-2]\!]$, $D^{q+1}\log Df$ is bounded on $(0,1)$ by $\norm{\log Df}_{R-1,[0,1]}$. Furthermore, $ X$ is continuous and thus bounded on $[0,1]$. Finally, $\mu_1=D X,\hdots, \mu_n=$ $ X^{n-1}D^n X$ are continuous on $(0,1)$ and extend continuously to $[0,1)$, by estimates $(i)_{2\,\text{to}\,n+1}$. Naturally, if $f$ is $C^R$ tangent to the identity at $1$, similar estimates hold near $1$. And if it is not, $ X$ is $C^{R-1}$ on a neighborhood of $1$, so $\mu_1$,..., $\mu_n$ are continuous there too. Thus, in the end, $\mu_1$,..., $\mu_n$ extend continuously to $[0,1]$, so they are bounded on $[0,1]$, and so is $Q_{n+1,q}(\mu_1,...,\mu_{n})$. 
\end{proof}
\bigskip

\noindent$(iii)_{n+1}$ : Note that
$$
L_X(\varphi_n\circ f^i)=D\varphi_n\circ f^i\times L_X f^i=D\varphi_n\circ f^i\times X \circ f^i=\varphi_{n+1}\circ f^i,
$$
so what we want to prove is that the symbol interversion $(*)$ below is licit:
$$
\Phi_{n+1}=L_X\Phi_n\underset{(iii)_{n}}{=}L_X\left(\sum_{i= 0}^{+\infty} \f_n\circ f^i \right)\underset{(*)}{=}\sum_{i= 0}^{+\infty} L_X(\f_n\circ f^i)=\sum_{i= 0}^{+\infty}\f_{n+1}\circ f^i.
$$
To do this, it is sufficient to prove that the last series converges uniformly on every segment contained in $(0,1)$. Let $J$ be such a segment. Let $C$ and $C'$ denote $\norm{\frac{X}{f-\id}}_{0,(0,1)}$ (cf. Lemma \ref{t:szek}) and $\norm{\frac{\varphi_{n+1}}{X}}_{0,(0,1)}$ (cf. Claim \ref{c:fibound}), respectively. For all $x\in J$, one has $f^i(x)\le x$ for all $i\in\N$, so 
\begin{equation}
\label{e:iini}
\left|\f_{n+1}(f^i(x)) \right|\le C' |X(f^i(x))|\le C'C\left(f^i(x)-f^{i+1}(x)\right).
\end{equation}
Therefore, since $\sum_{i\ge 0}(f^i-f^{i+1})$ converges uniformly on $J$ (towards $\id$), so does $\sum_{i\ge 0} \f_{n+1}\circ f^i$, and this concludes the proof of $(iii)_{n+1}$.
\bigskip

\noindent$(iv)_{n+1}$ : According to $(ii)_{n+1}$ and Lemma \ref{t:szek}, there exist $C>0$ and $x_1\in (0,1]$ such that, for all $x\in(0,x_1)$,
\begin{equation}
\label{e:iin}
|\f_{n+1}(x) |\le C (x-f(x))\times x^n\times \norm{f-\id}^{n+1-\frac{n+1}{R-1}}_{0,[0,x]}.
\end{equation}
Hence, for all such $x$,
\begin{align*}
|\Phi_{n+1}(x)|
&\le  \sum_{i\ge 0} |\f_{n+1}\circ f^i(x)|\quad\text{\,\, by $(iii)_{n+1}$ above,}\\
&\le C'x^n\norm{f-\id}^{n+1-\frac{n+1}{R-1}}_{0,[0,x]}\underbrace{\sum_{i\ge 0}(f^i(x)-f^{i+1}(x))}_{\le x}\quad\text{\,\, by \eqref{e:iin}.}
\end{align*}
This concludes the proof of $(iv)_{n+1}$, and thus the proof of Lemma~\ref{p:estb}.


\section{Appendix: Some tools for proving distortion from almost-reducibility in finite regularity}
\label{s:perfect}

In this section, we briefly present the tools behind the main statements of Section \ref{ss:a-r-dist} leading to Theorem \ref{t:dist-fini}, which claims that almost-reducibility implies distortion for compactly supported diffeomorphisms of $\R$, but with some loss of regularity. They are straightforward analogues in finite regularity of statements of \cite{EBM25}. We include them here for completeness's sake. 

\subsection{Ingredients of Proposition \ref{p:GSM2}}
\label{ss:tools1}

Proposition \ref{p:GSM2} is a direct consequence of the following two statements:

\begin{prop}
\label{p:GSM}
Let $r\in\N$ and $J,K$ be two bounded open intervals of $\R$ with $\bar J\subset K$. There exists a sequence $(\varepsilon_n)_n$ of positive real numbers such that the following property holds. For any sequences $(g_n)_{n \geq 0}$ and $(g'_n)_{n \geq 0}$ of $C^r$ diffeomorphisms supported in $J$ satisfying $d_{C^r}(g_n,\mathrm{id}) < \varepsilon_n$ and $d_{C^r}(g'_n,\mathrm{id})<\varepsilon_n$ for all $n$, there exists a finite subset $S$ of  $\mathrm{Diff}^{r}_c(J)$ such that, for each $n$, 
\begin{enumerate}
\item the commutator $[g_n,g'_n]$ belongs to $\langle S\rangle$,
\item $\ell_{S}([g_n,g'_n]) \leq 14 n+14$.
\end{enumerate}
\end{prop}

\begin{thm}
\label{t:unif-perfect}
Let $r\ge3$, $\varepsilon >0$ and let $I,J$ be two bounded open intervals of $\R$ with $\bar I\subset J$. There exists $\varepsilon'>0$ such that any diffeomorphism in $\mathrm{Diff}^{r}_c(I)$ which is $\varepsilon'$-close to the identity in $C^r$ topology coincides on $J$ with a product of five commutators of elements of $\mathrm{Diff}^{r-2}_c(J)$ that 
are $\varepsilon$-close to the identity in $C^{r-2}$ topology.
\end{thm}

\paragraph{Comments on Proposition \ref{p:GSM}.} The statement of Proposition \ref{p:GSM} is very classical and holds in a much more general context. Its proof is identical to that of Proposition 4.1 in \cite{EBM25} (for example), simply replacing $C^\infty$ by~$C^r$. Concretely, the family $S$ consists of four diffeomorphisms. Two of them depend only on $J$ and $K$: one supported in an intermediate interval $J'$ and for which $J$ is wandering, and one supported in $K$ and for which $J'$ is wandering. Naming $h$ the latter, the other two elements of $S$ are supported in $\cup_{n\in\N}h^n(J)$ and coincide on each component $h^n(J)$ with $h^ng_nh^{-n}$ and $h^ng_n'h^{-n}$ respectively.

\paragraph{Comments on Theorem \ref{t:unif-perfect}.} Just like Theorem 3.6 in \cite{EBM25}, we will see that it is a combination of an analogous statement for circle diffeomorphisms (Theorem \ref{Thm:locfragperf} below) and ideas of Mather already alluded to in Lemma \ref{l:cancel} and its applications. 

\begin{thm}[Local fragmented perfection] 
\label{Thm:locfragperf}
Let $r\ge3$ and $J_1,J_2$ be open intervals covering the circle. For any $\eta'>0$, there exists 
$\eta>0$ such that, for any  diffeomorphism $f \in \mathrm{Diff}^{r}_+(\mathbb{S}^1)$ with $d_{C^r}(f,\mathrm{id}) < \eta$, 
there exist two families of $C^{r-2}$-diffeomorphisms $(g_i)_{1 \leq i \leq 4}$ and $(g'_i)_{1 \leq i \leq 4}$ with the following properties:
\begin{enumerate}
\item $g_1,g'_1,g_4,g'_4$ are supported in $J_1$, while $g_2,g'_2,g_3,g'_3$ are supported in $J_2$.
\item For any $i$, one has $d_{C^{r-2}}(g_i,\mathrm{id}) <\eta'$ and $d_{C^{r-2}}(g'_i,\mathrm{id}) <\eta'$. 
\item $f=[g_1,g'_1][g_2,g'_2][g_3,g'_3][g_4,g'_4].$
\end{enumerate}
\end{thm}

This is a straightforward adaptation of Theorem 1.2 in \cite{Av08} (cf. also \cite[Thm 3.1]{EBM25}), which deals with the $C^\infty$ case. Before highlighting the changes between the finite and infinite regularity contexts, let us explain roughly how this statement implies Theorem \ref{t:unif-perfect} (for the details, see \cite{EBM25}; there are no subtleties due to finite regularity here). One starts with an auxiliary diffeomorphism $h$ which is $C^r$-close to the identity, supported in $J$, has no fixed point in $I$, has hyperbolic fixed points at $\partial I$ and is the time-$1$ map of a smooth flow. If $f\in\Diff^r_+(I)$ is sufficiently $C^r$-close to the identity, $fh$ still has no fixed points in $I$, it has the same derivatives as $h$ at $\partial I$, and it coincides with $h$ outside $I$. By Sternberg's linearization theorem (see \cite[Appendix]{Yoc} for the finite regularity version), the germs of $fh$ and $h$ at $0$ (resp. $1$) are $C^r$-conjugated by a germ which is the identity on the left of $0$ (resp. on the right of $1$). If these germs extend to a $C^r$ diffeomorphism $\f$ of the entire $I$ close to the identity and conjugating $h$ to $fh$, then $f$ can be written as a single commutator of ``small'' diffeomorphisms. This is not the case in general: $h_{|I}$ has a trivial Mather invariant while we only know that that of $fh$ is a circle diffeomorphism $C^r$-close to the identity. But thanks to Theorem \ref{Thm:locfragperf}, we can decompose it as a product of four commutators of circle diffeomorphisms $C^{r-2}$-close to the identity, supported in intervals of length less than $1$. Using the ideas of Lemma \ref{l:cancel}, one can then perform a ``correction'' on $fh$, composing it with some $u\in\Diff^{r-2}_c(\R)$ which is itself a product of four ``small'' commutators, so that $ufh$ is $C^{r-2}$ conjugated to $h$, which concludes the sketch of proof. 

\paragraph{Comments on Theorem \ref{Thm:locfragperf}}
The only difference between the proof of Theorem \ref{Thm:locfragperf} and Avila's original proof in $C^\infty$ regularity is the use of a finite regularity version of Herman's linearization theorem instead of the $C^\infty$ version. Namely:

\begin{thm}[Herman's theorem] 
\label{Thm:Herman}
Let $\xi$ be a diffeomorphism in $\mathrm{Diff}^{\infty}_+(\mathbb{S}^1)$ with ``very'' diophantine rotation number $\alpha$ (such that $|\alpha-\frac pq|>q^{-2-\delta}$ for some $\delta<1$, for all positive integers $p,q$. Given $r\ge3$, for any $\varepsilon_1>0$, there exists $\varepsilon_2>0$ 
such that, for any diffeomorphism $g \in \mathrm{Diff}^{r}_+(\mathbb{S}^1)$ with $\rho(g)=\rho(\xi)$ 
and $d_{C^r}(g,\xi)< \varepsilon_2$, there exists a diffeomorphism $h' \in \mathrm{Diff}^{r-2}_+(\mathbb{S}^1)$ such that $d_{C^{r-2}}(h',\mathrm{id})< \varepsilon_1$ and 
$$g=h' \xi (h')^{-1}.$$
\end{thm}
Apart from that, the proof of \ref{Thm:locfragperf} is identical to that of Avila, which has already been repeated in \cite{EBM25}, so we will not provide it here.

\subsection{Proof of Lemma \ref{Lem:smallpieces}}

Lemma \ref{Lem:smallpieces} is the finite regularity analogue of Lemma 4.3 in \cite{EBM25}. Just like the latter, it relies on the following particular case of a theorem of \cite{MR2509711} proved in $C^\infty$ regularity (cf. proof of Theorem 1.17 therein), but which actually holds in any regularity for which $\Diff^r_c(\R)$ 
is perfect, that is, for any $r\neq 2$. Recall that a conjugacy-invariant norm on a group $G$ is a map $\nu:G\to\R_+$ such that, for any $f,g\in G$, $\nu(f)=0$ if and only if $f=\id$, $\nu(fg)\le \nu(f)+\nu(g)$, $\nu(gfg^{-1})=\nu(f)$, and $\nu(f^{-1})=\nu(f)$.

\begin{thm}
\label{t:BIP}
Every conjugacy-invariant norm $\nu$ on $\Diff^r_c(\R)$, $r\neq2$, is bounded above by $14\nu(\phi)$ for any $\phi\in\Diff^r_c(\R)\setminus\{\id\}$. 
\end{thm}

This being said, the proof of Lemma \ref{Lem:smallpieces} is identical to that of its analogue in \cite{EBM25}, from which what follows is extracted, for completeness' sake. 

\begin{proof}[Proof of Lemma \ref{Lem:smallpieces}] 
If $f$ is the time-one map of a $C^{r}$ flow $(\varphi^{t})_{t \geq 0}$ on $\R$, then we can write $f= \left( \varphi^{\frac{1}{n}} \right)^{n}$, with sufficiently large $n$, so that 
properties 1 and 2 in Lemma \ref{Lem:smallpieces} hold by taking $S= \left\{\varphi^{\frac{1}{n}} \right\}$.

Now recall that for $r\neq2$, the group $G=\Diff^r_c(\R)$ is a simple group. As a conjugate of a flow by a diffeomorphism is still a flow, the subgroup of $G$ generated be time-one maps of $C^r$ flows is a nontrivial normal subgroup of $G$. Hence any diffeomorphism $f$ in $G$ can be written as a product of time-one maps of $C^r$ flows, and we denote by $\nu(f)$ the minimal number of factors of such a product. This defines a conjugation-invariant norm on the group $G$, which is thus bounded from above by $14\times1$ according to Theorem \ref{t:BIP}.
Hence, for any diffeomorphism $f$ in $\mathrm{Diff}^{r}_c(\R)$, there exists $k \leq 14$ and $C^{r}$ flows $(\varphi^{t}_i)_{t \in \mathbb{R}}$, for $i \in [\![1,k]\!]$, such that
$$f =\varphi_1^1 \varphi_2^1 \ldots \varphi_k^{1}.$$
Taking $n$ sufficiently large so that each diffeomorphism $\varphi^{\frac{1}{n}}_i$ is $\delta$-close to the identity in $C^r$ topology, we have  
$$f =\left(\varphi_1^{\frac{1}{n}}\right)^n \left(\varphi_2^{\frac{1}{n}}\right)^n \ldots \left(\varphi_k^{\frac{1}{n}}\right)^n,$$
and Properties 1 and 2 of the lemma hold with 
$$S= \left\{ \varphi_i^{\frac{1}{n}} \ | \ 1 \leq i \leq k \right\}.$$
Since $|S| \leq k \leq 14$, this closes the proof.
\end{proof}


\noindent{\bf Acknowledgments.} 
We are strongly grateful to Emmanuel Militon for several remarks and corrections to a first version of this article. 
This material is based upon work supported by the National Science Foundation under Grant No.~DMS-2424139, while the authors were in residence at the Simons Laufer Mathematical Sciences Institute in Berkeley, California, during the Spring 2026 semester. Both authors were also funded by the ECOS research project 230003 ``Small spaces under action'' during the last part of the preparation of the article. Moreover, the second-named author was funded by the Fondecyt regular research project 1260336 ``On the growth of certain geometric and dynamical structures''. 


\begin{footnotesize}

\bibliographystyle{amsalpha}
\bibliography{Biblio}

\noindent {\bf H\'el\`ene Eynard-Bontemps} \hfill{\bf Andr\'es Navas}

\noindent Sorbonne Université, Université Paris Cité,\hfill{ Dpto. de Matem\'aticas y C.C.}

\noindent  CNRS, IMJ-PRG,  \hfill{ Universidad de Santiago de  Chile}

\noindent F-75005 Paris, \hfill{Alameda Bernardo O'Higgins 3363}

\noindent France \hfill{Estaci\'on Central, Santiago, Chile} 

\noindent eynard@imj-prg.fr \hfill{andres.navas@usach.cl}

\end{footnotesize}

\end{document}